\documentclass[11pt]{amsart}
\usepackage{hyperref}
\usepackage[psamsfonts]{amssymb}
\usepackage{amsmath,amsfonts,amssymb,amscd}
\usepackage{latexsym,epsfig}
\usepackage[latin1]{inputenc}
\usepackage[T1]{fontenc}
\usepackage{enumerate}
\usepackage{amsbsy}
\usepackage{graphics}
\usepackage{graphicx}
\usepackage{grffile}
\usepackage{graphicx,subfigure,epsfig}
\usepackage{mathtools}
\usepackage[usenames,dvipsnames]{pstricks}

\DeclarePairedDelimiter\abs{\lvert}{\rvert}

\newtheorem{theorem}{Theorem}[section]

\newtheorem{lemma}[theorem]{Lemma}
\newtheorem{prop}[theorem]{Proposition}

\newtheorem{example}[theorem]{Example}
\newtheorem{remark}[theorem]{Remark}
\newtheorem{defi}[theorem]{Definition}

\begin{document}

\title[$F$-invariants of tabulated virtual knots]{$F$-polynomials of tabulated virtual knots}

\author{Maxim IVANOV}
\address{Laboratory of Topology and Dynamics, Novosibirsk State University, Novosibirsk, 630090, Russia}   \email{m.ivanov2@g.nsu.ru}

\author{Andrei VESNIN}

\address{Laboratory of Topology and Dynamics, Novosibirsk State University, Novosibirsk, 630090, Russia\\ 
Sobolev Institute of Mathematics, Novosibirsk, 630090, Russia\\ and Tomsk State University, Tomsk, 634050, Russia}
\email{vesnin@math.nsc.ru} 



\begin{abstract}
A sequence of $F$-polynomials $\{ F^n_K (t, \ell)\}_{n=1}^{\infty}$ of virtual knots $K$ was defined by Kaur, Prabhakar, and Vesnin in 2018.  These polynomials  have been expressed in terms of index value of crossing and $n$-writhe of $K$.  By the construction, $F$-polynomials are generalizations of the Kauffman's Affine Index Polynomial, and are invariants of virtual knot $K$. We present values of $F$-polynomials of oriented virtual knots having at most four classical crossings in a diagram.     
\end{abstract}

\thanks{This work was supported by the Laboratory of Topology and Dynamics, Novosibirsk State University (contract no. 14.Y26.31.0025 with the Ministry of Education and Science of the Russian Federation).}

\maketitle 

\section{Introduction}

Virtual knots were introduced by L.~Kauffman \cite{kauffman1999virtual} as a generalization of classical knots.  They are presented by virtual knot diagrams having classical crossings as well as virtual crossings. Equivalence  between two virtual knot diagrams can be determined through classical Reidemeister moves and virtual Reidemeister moves shown in Fig.~\ref{fig1a} and Fig.~\ref{fig1b}, respectively. 
\begin{figure}[!ht] 
\centering
\subfigure[Classical Reidemeister moves.]
{\includegraphics[scale=0.44]{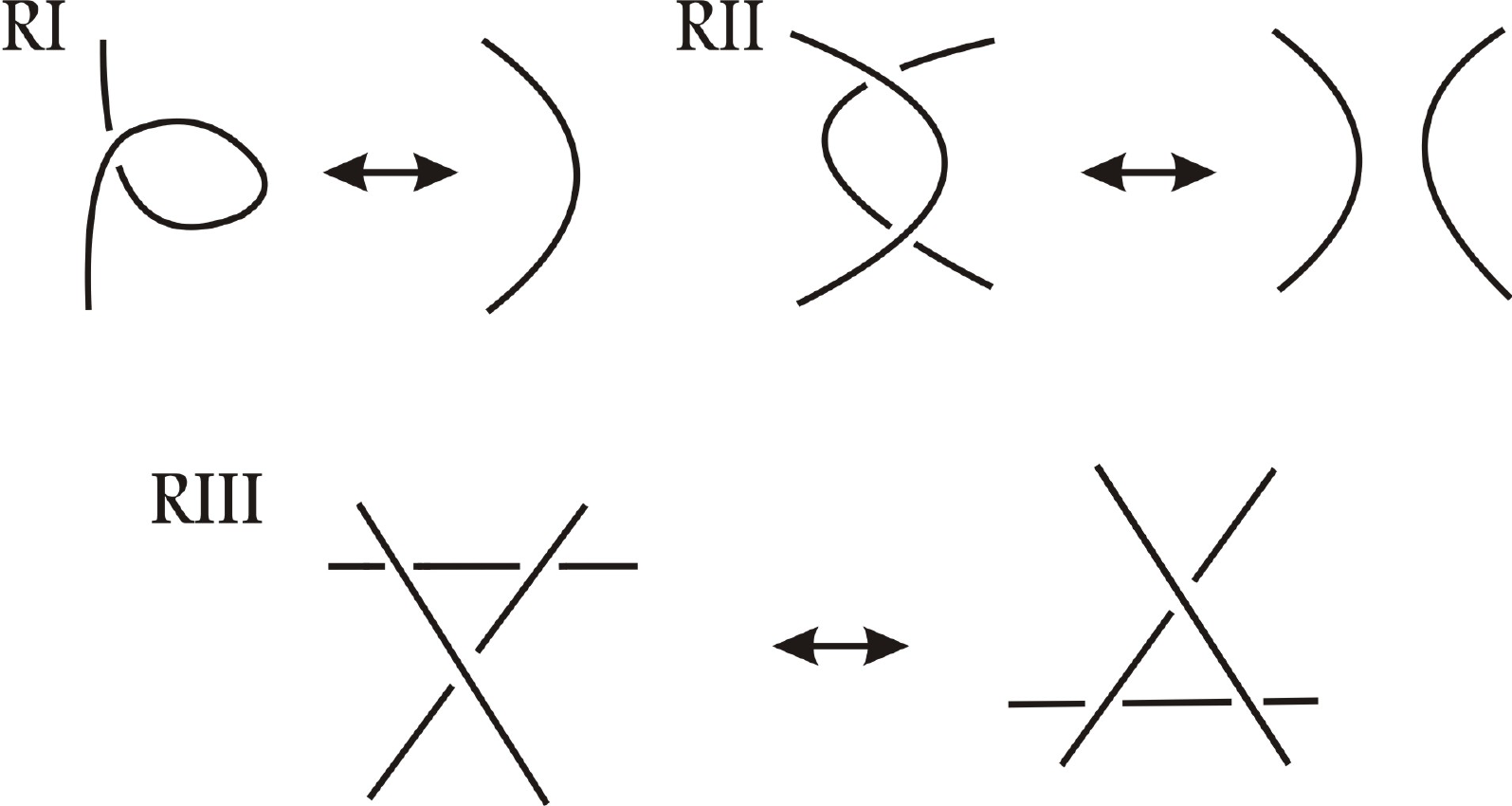} \label{fig1a}
} \medskip \subfigure[Virtual Reidemeister moves.]
{
\includegraphics[scale=.44]{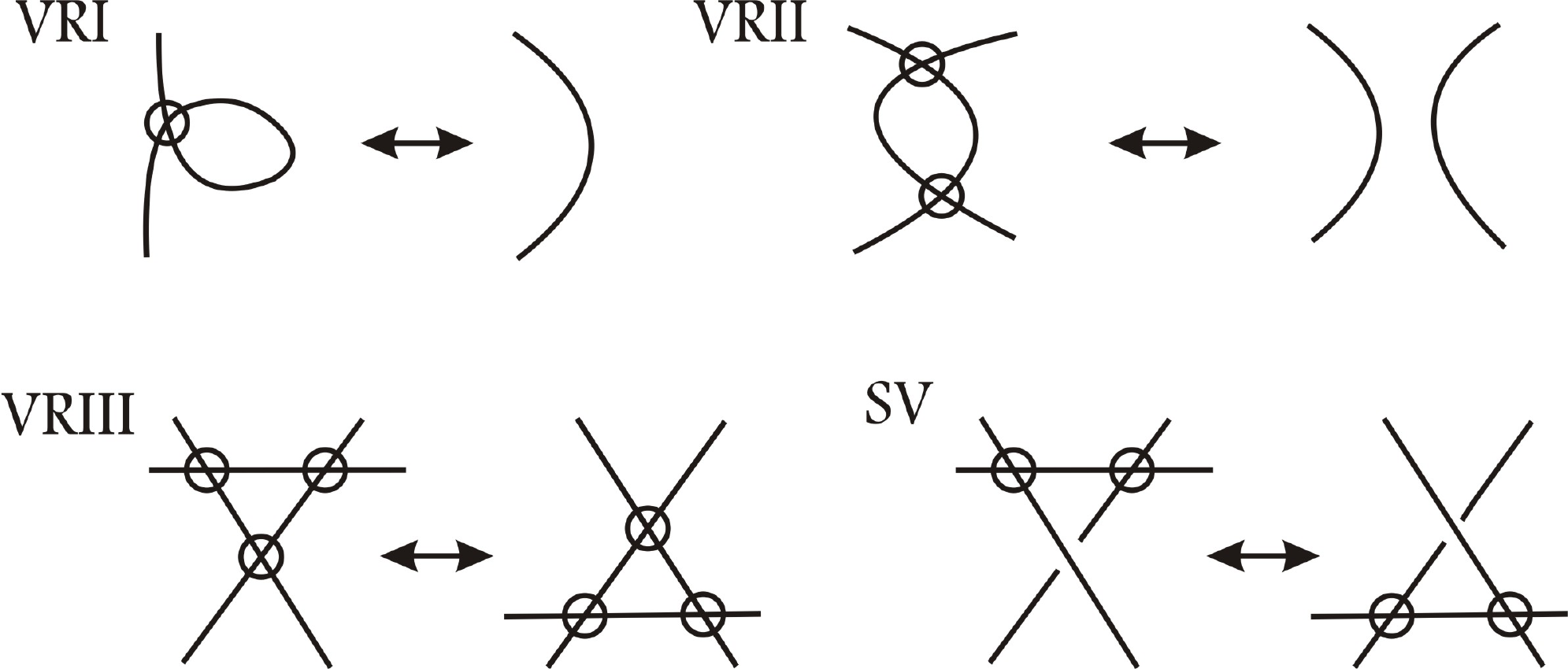} \label{fig1b}
}
\caption{Reidemeister moves.} 
\end{figure} 

Various invariants are known to distinguish two virtual knots. We are mainly interested in invariants of polynomial type. In the recent years, many polynomial invariants of virtual knots and links have been introduced. Among them are Affine Index Polynomial by L~Kauffman~\cite{kauffman2013affine}, Writhe Polynomial by Z.~Cheng and H.~Cao~\cite{cheng2013polynomial}, Wriggle Polynomial by L.~Folwaczny and L.~Kauffman~\cite{folwaczny2013linking}, Arrow Polynomial by H.~Dye and L.~Kauffman~\cite{dye2009virtual}, Extended Bracket Polynomial by L.~Kauffman~\cite{kauffman2009extended}, Index Polynomial by Y.-H.~Im, K.~Lee and S.-Y.~Lee~\cite{im2010index}, Zero Polynomial by M.-J.~Jeong~\cite{jeong2016zero},  sequences of $L$-polynomials and $F$-polynomials by K.~Kaur, M.~Prabhakar, and A.~Vesnin~\cite{KPV}. 

Let $K$ be an oriented virtual knot and $D$ be its diagram. For a positive integer $n$, in~\cite{KPV} $n$-th $F$-polynomial of $K$ was defined by assigning two weights for each classical crossing  $c \in D$. One is the \emph{index value} $\operatorname{Ind}(c)$, which was defined in~\cite{cheng2013polynomial}.  Second is the $n$-\emph{dwrithe} number $\nabla J _{n}(D)$, defined as difference between $n$-writhe and $(-n)$-writhe, with $n$-writhe defined in~\cite{satoh2014writhes}. For each classical crossing $c$ of the diagram $D$ we smooth it locally to obtain a virtual knot diagram $D_c$ with one less classical crossing. The smoothing rule, which we call a smoothing against orientation, is  shown below in Fig.~\ref{fig105}. After smoothing, we calculate $n$-dwrithe value $\nabla J_{n}(D_{c})$ of $D_c$ and assign it to the crossing $c$ of $D$.  An \emph{$n$-th $F$-polynomial} of oriented virtual knot $K$ is defined in~\cite{KPV} via its diagram $D$ as given in definition~\ref{f-pol} below.   

The paper is organized as follows. In Section~\ref{section-2} we give some basic definitions and known results on $F$-polynomials and explain the polynomial computations in details in Example~\ref{ex2.5}. Recall that virtual knots up to four crossings were  tabulated by Jeremy Green under the supervision of Dror Bar-Natan~\cite{Green} (see also Appendix A in the book~\cite{dye}). For the reader's convenience we present these diagrams in Tables~\ref{table-1}, \ref{table-2}, and \ref{table-3}, with the orientation indicated.The knots are denoted from \textbf{2.1} to \textbf{4.108}, where the first integer means the number of classical crossings. In Theorem~\ref{theorem3-1} $F$-polynomials of these virtual knots are given, see Tables~\ref{table-4}--\ref{table-8}.  In Proposition~\ref{prop3.2} we demonstrate that there exists an infinite family of oriented virtual knots with the same $F$-polynomials.

\section{Basic definitions} \label{section-2}

Let $D$ be an oriented virtual knot diagram. By an \emph{arc} we mean an edge between two consecutive classical crossings along the orientation.  The sign of classical crossing $c \in C(D)$, denoted by $\operatorname{sgn}(c)$, is defined as in Fig.~\ref{fig102}.  
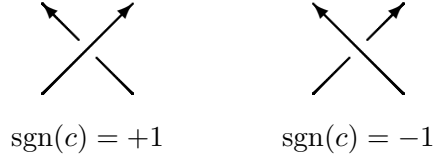
\begin{figure}[!ht]
\centering 
\unitlength=0.6mm
\begin{picture}(0,30)
\thicklines
\qbezier(-40,10)(-40,10)(-20,30)
\qbezier(-40,30)(-40,30)(-32,22) 
\qbezier(-20,10)(-20,10)(-28,18)
\put(-35,25){\vector(-1,1){5}}
\put(-25,25){\vector(1,1){5}}
\put(-30,0){\makebox(0,0)[cc]{$\operatorname{sgn}(c)=+1$}}
\qbezier(40,10)(40,10)(20,30)
\qbezier(40,30)(40,30)(32,22) 
\qbezier(20,10)(20,10)(28,18)
\put(25,25){\vector(-1,1){5}}
\put(35,25){\vector(1,1){5}}
\put(30,0){\makebox(0,0)[cc]{$\operatorname{sgn}(c)=-1$}}
\end{picture}
\caption{Crossing signs.} \label{fig102}
\end{figure}

Now assign an integer value to each arc in $D$ in such a way that the labeling around each crossing point of $D$ follows the rule as shown in Fig.~\ref{fig103}. L.~Kauffman proved in \cite[Proposition~4.1]{kauffman2013affine} that such integer labeling, called a \emph{Cheng coloring}, always exists for an oriented virtual knot diagram. Indeed, for an arc $\alpha$ of $D$ one can take label $\lambda (\alpha) = \sum_{c \in O(\alpha)} \operatorname{sgn} (c)$, where $O(\alpha)$ denotes the set of crossings first met as overcrossings on traveling along orientation, starting at the arc $\alpha$. 
\begin{figure}[!ht]
\centering 
\unitlength=0.6mm
\begin{picture}(0,30)(0,5)
\thicklines
\qbezier(-70,10)(-70,10)(-50,30)
\qbezier(-70,30)(-70,30)(-62,22) 
\qbezier(-50,10)(-50,10)(-58,18)
\put(-65,25){\vector(-1,1){5}}
\put(-55,25){\vector(1,1){5}}
\put(-75,34){\makebox(0,0)[cc]{$b+1$}}
\put(-75,8){\makebox(0,0)[cc]{$a$}}
\put(-45,8){\makebox(0,0)[cc]{$b$}}
\put(-45,34){\makebox(0,0)[cc]{$a-1$}}
\qbezier(10,10)(10,10)(-10,30)
\qbezier(10,30)(10,30)(2,22) 
\qbezier(-10,10)(-10,10)(-2,18)
\put(-5,25){\vector(-1,1){5}}
\put(5,25){\vector(1,1){5}}
\put(-15,34){\makebox(0,0)[cc]{$b+1$}}
\put(-15,8){\makebox(0,0)[cc]{$a$}}
\put(15,8){\makebox(0,0)[cc]{$b$}}
\put(15,34){\makebox(0,0)[cc]{$a-1$}}
\qbezier(70,10)(70,10)(50,30)
\qbezier(70,30)(70,30)(50,10) 
\put(55,25){\vector(-1,1){5}}
\put(65,25){\vector(1,1){5}}
\put(60,20){\circle{4}}
\put(45,34){\makebox(0,0)[cc]{$b$}}
\put(45,8){\makebox(0,0)[cc]{$a$}}
\put(75,8){\makebox(0,0)[cc]{$b$}}
\put(75,34){\makebox(0,0)[cc]{$a$}}
\end{picture}
\caption{Labeling around crossing.} \label{fig103}
\end{figure}
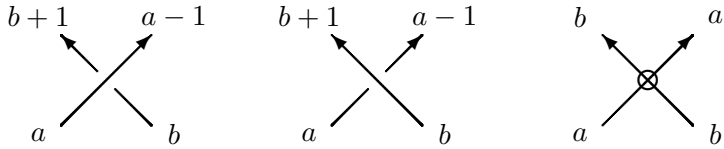

In \cite{cheng2013polynomial}, Z.~Cheng and H.~Gao assigned an integer value, called \emph{index value}, to each classical crossing $c$ of a virtual knot diagram and denoted it by $\operatorname{Ind}(c)$. It was proved  \cite[Theorem~3.6]{cheng2013polynomial} that  the following relation holds: 
\begin{equation}
\operatorname{Ind}(c) = \operatorname{sgn} (c)(a-b-1)  \label{eq2.2}
\end{equation}
with $a$ and $b$ be labels as presented in Fig.~\ref{fig103}. Recall that the Kauffman's affine index polynomial  can be defined as   
\begin{equation}
P_{D}(t) = \sum _{c\in C(D)} \operatorname{sgn}(c) (t^{\operatorname{Ind}(c)}-1), 
\end{equation}
where the summation runs over the set $C(D)$ of classical crossings of $D$.   

In \cite{satoh2014writhes}, S.~Satoh and K.~Taniguchi introduced the $n$-th writhe. For each $n \in \mathbb{Z}\setminus \{0\}$ the  \emph{$n$-th writhe $J_n(D)$} of an oriented virtual link diagram $D$ is defined as the number of positive sign crossings minus number of negative sign crossings of $D$ with index value $n$. Remark, that $J_n (D)$ is indeed coefficient of $t^n$ in the affine index polynomial. This $n$-th writhe is a virtual knot invariant, for more details we refer to \cite{satoh2014writhes}. Using $n$-th writhe, a new invariant was defined in~\cite{KPV} as follows. Let $n\in \mathbb{N}$ and $D$ be an oriented virtual knot diagram. Then the \emph{$n$-th dwrithe} of $D$, denoted by $\nabla J_{n}(D)$, is defined as 
 $$
 \nabla J_{n}(D)=J_{n}(D)-J_{-n}(D).
 $$

\smallskip 

\begin{remark} \label{rem2.1}
{\rm 
The $n$-th dwrithe  $\nabla J_{n}(D)$ is a virtual knot invariant, since  $n$-th writhe $J_n(D)$ is an oriented virtual knot invariant by~\cite{satoh2014writhes}.  Moreover, $\nabla J_{n}(D) =0$ for any classical knot diagram. 
}
\end{remark} 

\smallskip 
  
 \begin{remark} \label{rem2.2}
 {\rm 
Let $D$ be a virtual knot diagram. Consider the set of all affine index values  of crossing points:  
$$
S(D)=\{ \abs{ \operatorname{Ind}(c) } \, : \, c \in C(D)\} \subset \mathbb{N}. 
$$ 
Then $\nabla J_n(D)=0$ for any $n\in \mathbb{N}\setminus S(D)$. Therefore, for any virtual knot diagram $D$ there exists $n_{0}$ such that $\nabla J_n(D)=0$ for any $n$ greater than $n_{0}$.} 
\end{remark}
  
\smallskip 
     
Let $D^-$ be the \emph{reverse} of $D$, obtained from $D$ by reversing the orientation and let $D^*$ be the \emph{mirror image} of $D$, obtained by switching all the classical crossings in~$D$. 

\smallskip 

\begin{lemma} [\cite{KPV}] \label{lemma1}
If $D$ is an oriented virtual knot diagram, then $\nabla J_{n}(D^*)=\nabla J_{n}(D)$ and $\nabla J_{n}(D^-)=-\nabla J_{n}(D)$. 
\end{lemma} 

\smallskip 

Let $c$ be a classical crossing  of an oriented virtual knot diagram $D$. We consider two smoothings at $c$ as shown in Fig.~\ref{fig105}, depending on arcs orientations. Both smoothings will be referred to  as \emph{against orientation} smoothing.  
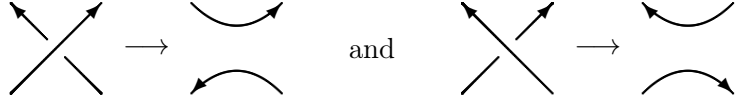
\begin{figure}[!ht]
\centering 
\unitlength=0.6mm
\begin{picture}(0,20)(0,10)
\thicklines
\qbezier(-80,10)(-80,10)(-60,30)
\qbezier(-80,30)(-80,30)(-72,22) 
\qbezier(-60,10)(-60,10)(-68,18)
\put(-75,25){\vector(-1,1){5}}
\put(-65,25){\vector(1,1){5}}
\put(-50,20){\makebox(0,0)[cc]{$\longrightarrow$}}
\qbezier(-40,30)(-30,20)(-20,30)
\qbezier(-40,10)(-30,20)(-20,10) 
\put(-35,15){\vector(-1,-1){5}}
\put(-25,25){\vector(1,1){5}}
\put(0,20){\makebox(0,0)[cc]{and}}
\qbezier(20,30)(20,30)(40,10)
\qbezier(20,10)(20,10)(28,18) 
\qbezier(40,30)(40,30)(32,22)
\put(25,25){\vector(-1,1){5}}
\put(35,25){\vector(1,1){5}}
\put(50,20){\makebox(0,0)[cc]{$\longrightarrow$}}
\qbezier(60,30)(70,20)(80,30)
\qbezier(60,10)(70,20)(80,10) 
\put(65,25){\vector(-1,1){5}}
\put(75,15){\vector(1,-1){5}}
\end{picture}
\caption{Smoothing against orientation.} \label{fig105}
\end{figure}

Let us denote by $D_{c}$ the oriented diagram obtained from $D$ by against orientation smoothing at $c$.  The orientation of $D_{c}$ is induced by the orientation of smoothing.  Since $D$ is a virtual knot diagram, $D_{c}$ is also a virtual knot diagram. 
 
\smallskip 

\begin{defi} [\cite{KPV}] \label{f-pol} 
{\rm Let $D$ be an oriented virtual knot diagram and $n$ be a positive integer. Then \emph{$n$-th  $F$-polynomial} of $D$ is defined as 
$$
\begin{gathered}
F_{D}^{n}(t,\ell) = \sum_{c \in C(D)} \operatorname{sgn}(c)t^{\text{Ind}(c)} \ell^{\nabla J_{n}(D_{c})}  
\qquad \qquad \qquad  \\  \qquad \qquad \qquad \qquad 
-  \sum _{c\in T_{n}(D)} \operatorname{sgn}(c) \ell^{\nabla J_{n}(D_{c})} - \sum _{c\notin T_{n}(D)} \operatorname{sgn} (c) \ell^{\nabla J_{n}(D)}, 
\end{gathered}
$$ 
where  $T_{n}(D)=\{c \in C(D) :  | \nabla J_{n}(D_{c}) | \, =  \, | \nabla J_{n}(D) | \}$. 
}
\end{defi}

We illustrate computation of $F$-polynomials in the following example. 

\begin{example} \label{ex2.5}
{\rm 
Let us consider an oriented virtual knot diagram $D = \textbf{3.1}$ presented in~Fig.~\ref{fig6}. 
\begin{figure}[!ht]
\centering 
\unitlength=0.34mm
\begin{picture}(0,100)
\put(-60,0){\begin{picture}(0,100)
\thicklines
\put(30,70){\vector(0,-1){5}}
\qbezier(-10,80)(-10,80)(10,100)
\qbezier(-10,100)(-10,100)(-3,93) 
\qbezier(3,87)(10,80)(10,80)
\qbezier(30,60)(30,60)(30,80)
\put(0,70){\circle{6}}
\qbezier(-10,60)(-10,60)(10,80)
\qbezier(-10,80)(-10,80)(10,60) 
\qbezier(30,40)(30,40)(30,60)
\qbezier(-10,60)(-10,60)(10,40)
\qbezier(-10,40)(-3,47)(-3,47)
\qbezier(3,53)(3,53)(10,60)
\qbezier(10,40)(10,40)(30,20)
\qbezier(10,20)(10,20)(30,40)
\qbezier(-10,20)(-10,20)(-10,40)
\put(20,30){\circle{6}}
\qbezier(-10,20)(-10,20)(-3,13)
\qbezier(3,7)(10,0)(10,0)
\qbezier(-10,0)(-10,0)(10,20)
\qbezier(10,0)(30,0)(30,20)
\qbezier(-10,0)(-30,0)(-30,20)
\qbezier(-30,20)(-30,20)(-30,80)
\qbezier(-30,80)(-30,100)(-10,100)
\qbezier(30,80)(30,100)(10,100)
\put(-15,90){\makebox(0,0)[cc]{\footnotesize $\alpha$}}
\put(-15,50){\makebox(0,0)[cc]{\footnotesize $\beta$}}
\put(-15,10){\makebox(0,0)[cc]{\footnotesize $\gamma$}}
\end{picture}} 
\put(60,0){\begin{picture}(0,100)
\thicklines
\put(30,70){\vector(0,-1){5}}
\qbezier(-10,80)(-10,80)(10,100)
\qbezier(-10,100)(-10,100)(-3,93) 
\qbezier(3,87)(10,80)(10,80)
\qbezier(30,60)(30,60)(30,80)
\put(0,70){\circle{6}}
\qbezier(-10,60)(-10,60)(10,80)
\qbezier(-10,80)(-10,80)(10,60) 
\qbezier(30,40)(30,40)(30,60)
\qbezier(-10,60)(-10,60)(10,40)
\qbezier(-10,40)(-3,47)(-3,47)
\qbezier(3,53)(3,53)(10,60)
\qbezier(10,40)(10,40)(30,20)
\qbezier(10,20)(10,20)(30,40)
\qbezier(-10,20)(-10,20)(-10,40)
\put(20,30){\circle{6}}
\qbezier(-10,20)(-10,20)(-3,13)
\qbezier(3,7)(10,0)(10,0)
\qbezier(-10,0)(-10,0)(10,20)
\qbezier(10,0)(30,0)(30,20)
\qbezier(-10,0)(-30,0)(-30,20)
\qbezier(-30,20)(-30,20)(-30,80)
\qbezier(-30,80)(-30,100)(-10,100)
\qbezier(30,80)(30,100)(10,100)
\put(-35,50){\makebox(0,0)[cc]{\footnotesize $1$}}
\put(35,80){\makebox(0,0)[cc]{\footnotesize $0$}}
\put(-20,80){\makebox(0,0)[cc]{\footnotesize $-1$}}
\put(15,80){\makebox(0,0)[cc]{\footnotesize $0$}}
\put(-15,60){\makebox(0,0)[cc]{\footnotesize $0$}}
\put(20,60){\makebox(0,0)[cc]{\footnotesize $-1$}}
\put(20,40){\makebox(0,0)[cc]{\footnotesize $-1$}}
\put(-20,30){\makebox(0,0)[cc]{\footnotesize $-2$}}
\put(40,15){\makebox(0,0)[cc]{\footnotesize $-1$}}
\put(10,15){\makebox(0,0)[cc]{\footnotesize $0$}}
\put(-3,7){\vector(-1,-1){5}}
\put(-3,13){\vector(-1,1){5}}
\put(3,53){\vector(1,1){5}}
\put(3,47){\vector(1,-1){5}}
\put(3,93){\vector(1,1){5}}
\put(3,87){\vector(1,-1){5}}
\end{picture}} 
\end{picture}
\caption{Oriented virtual diagram $D = \textbf{3.1}$ and its labeling.} \label{fig5}
\end{figure}
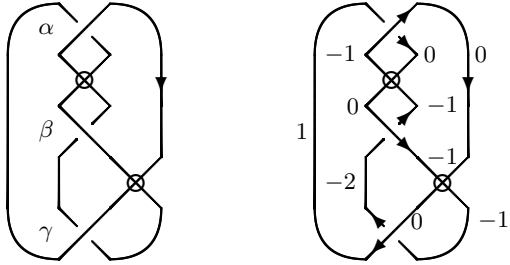
The diagram $D$ has three classical crossings denoted by $\alpha$, $\beta$, and $\gamma$, as it is shown in the left-hand picture. In the right-hand picture we presented orientation of arcs for each classical crossing and the corresponding labeling, satisfying the rule given in Fig.~\ref{fig103}.  Crossing signs can easily be found from arc orientations around crossing points given in Fig.~\ref{fig5}: 
$\operatorname{sgn}(\alpha) = \operatorname{sgn}(\gamma) =-1$  and $\operatorname{sgn}(\beta)=1$. 
Index values can be calculated directly from crossing signs and labeling of arcs by Eq.~(\ref{eq2.2}): 
$\operatorname{Ind}(\alpha)= - 1$,  $\operatorname{Ind}(\beta)=1$ and $\operatorname{Ind}(\gamma)=2$. 
Therefore, only the following writhe numbers can be non-trivial: $J_1(D)$, $J_{-1}(D)$, and $J_2(D)$. It is easy to see, that $J_1(D) = \operatorname{sgn} (\beta) = 1$, $J_{-1} (D) = \operatorname{sgn}(\alpha) = -1$, and $J_{2} (D) = \operatorname{sgn}(\gamma) = -1$. Then $\nabla J_1(D) = J_1(D) - J_{-1}(D) = 2$ and $\nabla J_2 (D) = J_2(D) - J_{-2}(D) =-1$. For any $n \geq 3$ we have $\nabla J_n (D) =0$.  

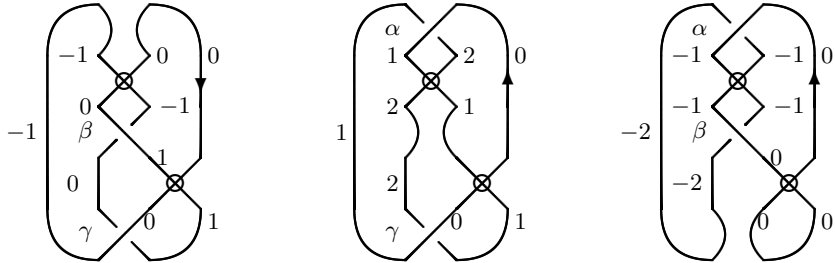
\begin{figure}[!ht]
\centering 
\unitlength=0.34mm
\begin{picture}(0,100)(0,0)
\put(-120,0){\begin{picture}(0,100)
\thicklines
\put(30,70){\vector(0,-1){5}}
\qbezier(-10,80)(0,90)(-10,100)
\qbezier(10,80)(0,90)(10,100) 
\qbezier(30,60)(30,60)(30,80)
\put(0,70){\circle{6}}
\qbezier(-10,60)(-10,60)(10,80)
\qbezier(-10,80)(-10,80)(10,60) 
\qbezier(30,40)(30,40)(30,60)
\qbezier(-10,60)(-10,60)(10,40)
\qbezier(-10,40)(-3,47)(-3,47)
\qbezier(3,53)(3,53)(10,60)
\qbezier(10,40)(10,40)(30,20)
\qbezier(10,20)(10,20)(30,40)
\qbezier(-10,20)(-10,20)(-10,40)
\put(20,30){\circle{6}}
\qbezier(-10,20)(-10,20)(-3,13)
\qbezier(3,7)(10,0)(10,0)
\qbezier(-10,0)(-10,0)(10,20)
\qbezier(10,0)(30,0)(30,20)
\qbezier(-10,0)(-30,0)(-30,20)
\qbezier(-30,20)(-30,20)(-30,80)
\qbezier(-30,80)(-30,100)(-10,100)
\qbezier(30,80)(30,100)(10,100)
\put(-15,50){\makebox(0,0)[cc]{\footnotesize $\beta$}}
\put(-15,10){\makebox(0,0)[cc]{\footnotesize $\gamma$}}
\put(-40,50){\makebox(0,0)[cc]{\footnotesize $-1$}}
\put(35,80){\makebox(0,0)[cc]{\footnotesize $0$}}
\put(-20,80){\makebox(0,0)[cc]{\footnotesize $-1$}}
\put(15,80){\makebox(0,0)[cc]{\footnotesize $0$}}
\put(-15,60){\makebox(0,0)[cc]{\footnotesize $0$}}
\put(20,60){\makebox(0,0)[cc]{\footnotesize $-1$}}
\put(15,40){\makebox(0,0)[cc]{\footnotesize $1$}}
\put(-20,30){\makebox(0,0)[cc]{\footnotesize $0$}}
\put(35,15){\makebox(0,0)[cc]{\footnotesize $1$}}
\put(10,15){\makebox(0,0)[cc]{\footnotesize $0$}}
\end{picture}} 
\put(0,0){\begin{picture}(0,100)
\thicklines
\put(30,70){\vector(0,1){5}}
\qbezier(-10,80)(-10,80)(10,100)
\qbezier(-10,100)(-10,100)(-3,93) 
\qbezier(3,87)(10,80)(10,80)
\qbezier(30,60)(30,60)(30,80)
\put(0,70){\circle{6}}
\qbezier(-10,60)(-10,60)(10,80)
\qbezier(-10,80)(-10,80)(10,60) 
\qbezier(30,40)(30,40)(30,60)
\qbezier(-10,60)(0,50)(-10,40)
\qbezier(10,40)(0,50)(10,60)
\qbezier(10,40)(10,40)(30,20)
\qbezier(10,20)(10,20)(30,40)
\qbezier(-10,20)(-10,20)(-10,40)
\put(20,30){\circle{6}}
\qbezier(-10,20)(-10,20)(-3,13)
\qbezier(3,7)(10,0)(10,0)
\qbezier(-10,0)(-10,0)(10,20)
\qbezier(10,0)(30,0)(30,20)
\qbezier(-10,0)(-30,0)(-30,20)
\qbezier(-30,20)(-30,20)(-30,80)
\qbezier(-30,80)(-30,100)(-10,100)
\qbezier(30,80)(30,100)(10,100)
\put(-15,90){\makebox(0,0)[cc]{\footnotesize $\alpha$}}
\put(-15,10){\makebox(0,0)[cc]{\footnotesize $\gamma$}}
\put(-35,50){\makebox(0,0)[cc]{\footnotesize $1$}}
\put(35,80){\makebox(0,0)[cc]{\footnotesize $0$}}
\put(-15,80){\makebox(0,0)[cc]{\footnotesize $1$}}
\put(15,80){\makebox(0,0)[cc]{\footnotesize $2$}}
\put(-15,60){\makebox(0,0)[cc]{\footnotesize $2$}}
\put(15,60){\makebox(0,0)[cc]{\footnotesize $1$}}
\put(-15,30){\makebox(0,0)[cc]{\footnotesize $2$}}
\put(35,15){\makebox(0,0)[cc]{\footnotesize $1$}}
\put(10,15){\makebox(0,0)[cc]{\footnotesize $0$}}
\end{picture}} 
\put(120,0){\begin{picture}(0,100)
\thicklines
\put(30,70){\vector(0,1){5}}
\qbezier(-10,80)(-10,80)(10,100)
\qbezier(-10,100)(-10,100)(-3,93) 
\qbezier(3,87)(10,80)(10,80)
\qbezier(30,60)(30,60)(30,80)
\put(0,70){\circle{6}}
\qbezier(-10,60)(-10,60)(10,80)
\qbezier(-10,80)(-10,80)(10,60) 
\qbezier(30,40)(30,40)(30,60)
\qbezier(-10,60)(-10,60)(10,40)
\qbezier(-10,40)(-3,47)(-3,47)
\qbezier(3,53)(3,53)(10,60)
\qbezier(10,40)(10,40)(30,20)
\qbezier(10,20)(10,20)(30,40)
\qbezier(-10,20)(-10,20)(-10,40)
\put(20,30){\circle{6}}
\qbezier(-10,20)(0,10)(-10,0)
\qbezier(10,20)(0,10)(10,0)
\qbezier(10,0)(30,0)(30,20)
\qbezier(-10,0)(-30,0)(-30,20)
\qbezier(-30,20)(-30,20)(-30,80)
\qbezier(-30,80)(-30,100)(-10,100)
\qbezier(30,80)(30,100)(10,100)
\put(-15,90){\makebox(0,0)[cc]{\footnotesize $\alpha$}}
\put(-15,50){\makebox(0,0)[cc]{\footnotesize $\beta$}}
\put(-40,50){\makebox(0,0)[cc]{\footnotesize $-2$}}
\put(35,80){\makebox(0,0)[cc]{\footnotesize $0$}}
\put(-20,80){\makebox(0,0)[cc]{\footnotesize $-1$}}
\put(20,80){\makebox(0,0)[cc]{\footnotesize $-1$}}
\put(-20,60){\makebox(0,0)[cc]{\footnotesize $-1$}}
\put(20,60){\makebox(0,0)[cc]{\footnotesize $-1$}}
\put(15,40){\makebox(0,0)[cc]{\footnotesize $0$}}
\put(-20,30){\makebox(0,0)[cc]{\footnotesize $-2$}}
\put(35,15){\makebox(0,0)[cc]{\footnotesize $0$}}
\put(10,15){\makebox(0,0)[cc]{\footnotesize $0$}}
\end{picture}} 
\end{picture}
\caption{Oriented virtual diagrams $D_{\alpha}$, $D_{\beta}$, and $D_{\gamma}$.} \label{fig6}
\end{figure}

The result of calculations for oriented virtual knot diagrams $D_{\alpha}$, $D_{\beta}$, and $D_{\gamma}$, presented in Fig.~\ref{fig6}, is given in Table~\ref{t1}. 
\begin{table}[ht]
\caption{Values of $\operatorname{sgn}$, $\operatorname{Ind}$, and dwrithe for diagrams from Fig.~\ref{fig6}.}
\begin{tabular}{l l l l } \hline
 & Sign & index value & dwrihe \\ \hline
$D_{\alpha}$ & $\operatorname{sgn}(\beta) = 1$ & $\operatorname{Ind}(\beta)=1$ & $\nabla J_{1}(D_{\alpha}) =  0$ \\ 
& $\operatorname{sgn}(\gamma)=1$ &  $\operatorname{Ind}(\gamma)=-1$& $\nabla J_{2}(D_{\alpha}) = 0$ \\  \hline
$D_{\beta}$ & $\operatorname{sgn}(\alpha)=-1$ & $\operatorname{Ind}(\alpha)=-1$ & $\nabla J_{1}(D_{\beta})=0$ \\ 
& $\operatorname{sgn}(\gamma)=1$ & $\operatorname{Ind}(\gamma)=-1$ & $\nabla J_{2}(D_{\beta}) = 0$   \\ \hline
$D_{\gamma}$ & $\operatorname{sgn}(\alpha)=1$ & $\operatorname{Ind}(\alpha)= 1$ &  $\nabla J_{1}(D_{\gamma})= 0$ \\ 
& $\operatorname{sgn}(\beta)=-1$ & $\operatorname{Ind}(\beta)= 1$ & $\nabla J_{2}(D_{\gamma})=0$ \\ \hline
\end{tabular} \label{t1}
\end{table} 
Basing on these calculations we obtain that $T_{1} (D) = \emptyset$, $T_{2} (D) = \emptyset$,  and $F$-polynomials for diagram $D$. Namely, for $n = 1$ and $n=2$ we get 
$$
F^{1}_{D}(t,\ell) = -t^{-1} + t - t^{2} + \ell^{2}, \qquad F^{2}_{D}(t,\ell)= -t^{-1} + t - t^{2} + \ell^{-1}.   
$$
For all $n \geq 3$ $F$-polynomials coincide with the Affine Index Polynomial:  
$$
F^{n}_{D}(t,\ell) = P_D (t) = -t^{-1} + t - t^{2} +1.
$$
Through calculations of $F$-polynomial we considered oriented virtual knot diagram $D$ of $\textbf{3.1}$ with the orientation given in Figure~ \ref{fig6}.  The $F$-polynomials of $\textbf{3.1}$ presented in $\ref{table-4}$ were done for the opposite orientation. 
We recall that changing orientation of the diagram leads to changing the sign of the dwrith number. This completes the Example~\ref{ex2.5}. 
}
\end{example}


\begin{theorem} [\cite{KPV}] \label{fth} 
For any positive integer $n$ the polynomial $F^{n}_{K}(t,\ell)$ is an oriented virtual knot invariant. 
\end{theorem}
 
\smallskip  

\section{$F$-polynomials for Green's table and for Kauffman's infinite family}

Tabulation of prime classical knot diagrams  in order of  increasing of crossing numbers was started in 19-th century. Nowdays tables of classical knot diagrams as well as their invariants are widely presented in the literature. 

Unfortunately, there are few results related to tabulation of diagrams of virtual knots. We point out that virtual knots up to four crossings were tabulated by Jeremy Green under the supervision of Dror Bar-Natan~\cite{Green} (see also Appendix A in the book~\cite{dye}). For the reader's convenience we present these diagrams in Tables~\ref{table-1}, \ref{table-2}, and \ref{table-3}, the knots are denoted from \textbf{2.1} to \textbf{4.108}. 

We computed $F^{n}$-polynomials for these knots and splitted up knots into groups with the same polynomials. For each knot $K$ polynomials $F^{n}$ are presented up to such  $n$ that $F^{n}_{K} (t,\ell)$ reduces to Affine Index Polynomial $P_{K}(t)$. 

\smallskip

\begin{theorem} \label{theorem3-1}
$F$-polynomials of virtual knots with diagrams from  \textbf{2.1} to \textbf{4.108} are presented in Tables~\ref{table-4}--\ref{table-8}. 
\end{theorem}

\begin{proof} 
The result is obtained by  computer calculations which follow the algorithm of finding $F$-polynomials based on the definition. 
\end{proof}

\smallskip

Observe that since  \textbf{3.6} is the classical trefoil and \textbf{4.108} is the classical figure-eight knot, their $F$-polynomials are trivial.

 \smallskip 
 
 We recall that in~\cite[Fig.~27]{kauffman2018cobordism} Kauffman presented an infinite family of virtual knots with the same Affine Index Polynomial. He mentioned that all of them are labeled cobordant to the Hopf link diagram and they can be distinguished from one another by the bracket polynomial. 
 
 \begin{prop} \label{prop3.2} 
 There exists an infinite family of oriented virtual knots with the same $F$-polynomials. 
 \end{prop}
 
 \begin{proof} We will use the same family of virtual knots as Kauffman used  in \cite{kauffman2018cobordism} for the Affine Index Polynomial. 
 Consider oriented virtual knot diagram $D^{1}$ with one virtual crossings and three classical crossings $\alpha_{1}$, $\beta$ and $\gamma$ as presented in Fig.~\ref{fig7}. Applying smoothings against orientation in classical crossings, we will obtain three oriented virtual knot diagrams $D^{1}_{\alpha_{1}}$, $D^{1}_{\beta}$ and $D^{1}_{\gamma}$ pictured in Fig.~\ref{fig7}. 
 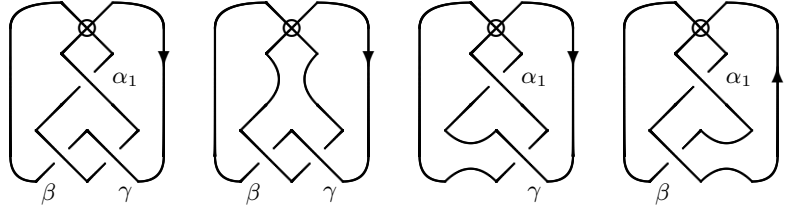
\begin{figure}[!ht]
\centering 
\unitlength=0.34mm
\begin{picture}(0,90)(0,10)
\put(-120,0){\begin{picture}(0,80)
\thicklines
\put(30,60){\vector(0,-1){5}}
\qbezier(-10,60)(-10,60)(10,80)
\qbezier(-10,80)(-10,80)(10,60) 
\put(0,70){\circle{6}}
\qbezier(-10,60)(-10,60)(10,40)
\qbezier(-10,40)(-10,40)(-3,47) 
\qbezier(10,60)(10,60)(3,53) 
\qbezier(-10,40)(-10,40)(-20,30)
\qbezier(10,40)(10,40)(20,30)
\qbezier(-20,30)(-20,30)(0,10)
\qbezier(-20,10)(-20,10)(-13,17) 
\qbezier(0,30)(0,30)(-7,23)
\qbezier(0,30)(0,30)(20,10)
\qbezier(0,10)(0,10)(7,17)
\qbezier(20,30)(20,30)(13,23)
\qbezier(20,10)(30,10)(30,20)
\qbezier(30,20)(30,20)(30,70)
\qbezier(30,70)(30,80)(10,80)
\qbezier(-20,10)(-30,10)(-30,20)
\qbezier(-30,20)(-30,20)(-30,70)
\qbezier(-30,70)(-30,80)(-10,80)
\put(15,50){\makebox(0,0)[cc]{\footnotesize $\alpha_{1}$}}
\put(-15,5){\makebox(0,0)[cc]{\footnotesize $\beta$}}
\put(15,5){\makebox(0,0)[cc]{\footnotesize $\gamma$}}
\end{picture}} 
\put(-40,0){\begin{picture}(0,80)
\thicklines
\put(30,60){\vector(0,-1){5}}
\qbezier(-10,60)(-10,60)(10,80)
\qbezier(-10,80)(-10,80)(10,60) 
\put(0,70){\circle{6}}
\qbezier(-10,40)(0,50)(-10,60)
\qbezier(10,40)(0,50)(10,60)
\qbezier(-10,40)(-10,40)(-20,30)
\qbezier(10,40)(10,40)(20,30)
\qbezier(-20,30)(-20,30)(0,10)
\qbezier(-20,10)(-20,10)(-13,17) 
\qbezier(0,30)(0,30)(-7,23)
\qbezier(0,30)(0,30)(20,10)
\qbezier(0,10)(0,10)(7,17)
\qbezier(20,30)(20,30)(13,23)
\qbezier(20,10)(30,10)(30,20)
\qbezier(30,20)(30,20)(30,70)
\qbezier(30,70)(30,80)(10,80)
\qbezier(-20,10)(-30,10)(-30,20)
\qbezier(-30,20)(-30,20)(-30,70)
\qbezier(-30,70)(-30,80)(-10,80)
%
\put(-15,5){\makebox(0,0)[cc]{\footnotesize $\beta$}}
\put(15,5){\makebox(0,0)[cc]{\footnotesize $\gamma$}}
\end{picture}} 
\put(40,0){\begin{picture}(0,80)
\thicklines
\put(30,60){\vector(0,-1){5}}
\qbezier(-10,60)(-10,60)(10,80)
\qbezier(-10,80)(-10,80)(10,60) 
\put(0,70){\circle{6}}
\qbezier(-10,60)(-10,60)(10,40)
\qbezier(-10,40)(-10,40)(-3,47) 
\qbezier(10,60)(10,60)(3,53) 
\qbezier(-10,40)(-10,40)(-20,30)
\qbezier(10,40)(10,40)(20,30)
\qbezier(-20,30)(-10,20)(0,30)
\qbezier(-20,10)(-10,20)(0,10)
\qbezier(0,30)(0,30)(20,10)
\qbezier(0,10)(0,10)(7,17)
\qbezier(20,30)(20,30)(13,23)
\qbezier(20,10)(30,10)(30,20)
\qbezier(30,20)(30,20)(30,70)
\qbezier(30,70)(30,80)(10,80)
\qbezier(-20,10)(-30,10)(-30,20)
\qbezier(-30,20)(-30,20)(-30,70)
\qbezier(-30,70)(-30,80)(-10,80)
\put(15,50){\makebox(0,0)[cc]{\footnotesize $\alpha_{1}$}}
\put(15,5){\makebox(0,0)[cc]{\footnotesize $\gamma$}}
\end{picture}} 
\put(120,0){\begin{picture}(0,80)
\thicklines
\put(30,50){\vector(0,1){5}}
\qbezier(-10,60)(-10,60)(10,80)
\qbezier(-10,80)(-10,80)(10,60) 
\put(0,70){\circle{6}}
\qbezier(-10,60)(-10,60)(10,40)
\qbezier(-10,40)(-10,40)(-3,47) 
\qbezier(10,60)(10,60)(3,53) 
\qbezier(-10,40)(-10,40)(-20,30)
\qbezier(10,40)(10,40)(20,30)
\qbezier(-20,30)(-20,30)(0,10)
\qbezier(-20,10)(-20,10)(-13,17) 
\qbezier(0,30)(0,30)(-7,23)
\qbezier(0,30)(10,20)(20,30)
\qbezier(0,10)(10,20)(20,10)
\qbezier(20,10)(30,10)(30,20)
\qbezier(30,20)(30,20)(30,70)
\qbezier(30,70)(30,80)(10,80)
\qbezier(-20,10)(-30,10)(-30,20)
\qbezier(-30,20)(-30,20)(-30,70)
\qbezier(-30,70)(-30,80)(-10,80)
\put(15,50){\makebox(0,0)[cc]{\footnotesize $\alpha_{1}$}}
\put(-15,5){\makebox(0,0)[cc]{\footnotesize $\beta$}}
\end{picture}} 
\end{picture}
\caption{Oriented virtual diagrams $D^{1}$, $D^{1}_{\alpha_{1}}$, $D^{1}_{\beta}$, and $D^{1}_{\gamma}$.} \label{fig7}
\end{figure}
It is easy to calculate the values presented in Table~\ref{t100}.  
\begin{table}[ht]
\caption{Values of $\operatorname{sgn}$, $\operatorname{Ind}$, and dwrithe for diagrams from Fig.~\ref{fig7}.}
\begin{tabular}{l l l l } \hline
 & Sign & index value & dwrithe \\ \hline
 $D^{1}$ & $\operatorname{sgn}(\alpha_{1}) = 1$ & $\operatorname{Ind}(\alpha_{1})=0$ & $\nabla J_{1}(D^{1}) =  0$ \\ 
& $\operatorname{sgn}(\beta)=-1$ &  $\operatorname{Ind}(\beta)=1$& \\  
& $\operatorname{sgn}(\gamma)=-1$ &  $\operatorname{Ind}(\gamma)=-1$& \\  \hline
$D^{1}_{\alpha_{1}}$ & $\operatorname{sgn}(\beta) = 1$ & $\operatorname{Ind}(\beta)=1$ & $\nabla J_{1}(D^{1}_{\alpha_{1}}) =  0$ \\ 
& $\operatorname{sgn}(\gamma)=1$ &  $\operatorname{Ind}(\gamma)=-1$& \\  \hline
$D^{1}_{\beta}$ & $\operatorname{sgn}(\alpha_{1})=-1$ & $\operatorname{Ind}(\alpha_{1})=-1$ & $\nabla J_{1}(D^{1}_{\beta})=0$ \\ 
& $\operatorname{sgn}(\gamma)=-1$ & $\operatorname{Ind}(\gamma)=1$ &   \\ \hline
$D^{1}_{\gamma}$ & $\operatorname{sgn}(\alpha_{1})=1$ & $\operatorname{Ind}(\alpha_{1})= 1$ &  $\nabla J_{1}(D^{1}_{\gamma})= 0$ \\ 
& $\operatorname{sgn}(\beta)=-1$ & $\operatorname{Ind}(\beta)= -1$ &  \\ \hline
\end{tabular} \label{t100}
\end{table} 
 Therefore $T_{1} = \{ \alpha_{1}, \beta, \gamma \}$ and $F^{1}_{D^{1}}(t,\ell) = P_{D^{1}} (t)= -t + 2 -t^{-1}$. 
 
 Similar to~\cite[Fig.~27]{kauffman2018cobordism} diagram $D^{1}$ can be naturally generalized to diagram $D^{k}$ with an odd number $k$ of classical crossings as shown in Fig.~\ref{fig8} for $k=3$. 
  \begin{figure}[!ht]
\centering 
\unitlength=0.34mm
\begin{picture}(0,130)(0,10)
\put(0,0){\begin{picture}(0,120)
\thicklines
\put(30,60){\vector(0,-1){5}}
\qbezier(-10,100)(-10,100)(10,120)
\qbezier(-10,120)(-10,120)(10,100) 
\put(0,110){\circle{6}}
\qbezier(-10,100)(-10,100)(10,80)
\qbezier(-10,80)(-10,80)(-3,87) 
\qbezier(10,100)(10,100)(3,93) 
\qbezier(-10,80)(-10,80)(10,60)
\qbezier(-10,60)(-10,60)(-3,67) 
\qbezier(10,80)(10,80)(3,73) 
\qbezier(-10,60)(-10,60)(10,40)
\qbezier(-10,40)(-10,40)(-3,47) 
\qbezier(10,60)(10,60)(3,53) 
\qbezier(-10,40)(-10,40)(-20,30)
\qbezier(10,40)(10,40)(20,30)
\qbezier(-20,30)(-20,30)(0,10)
\qbezier(-20,10)(-20,10)(-13,17) 
\qbezier(0,30)(0,30)(-7,23)
\qbezier(0,30)(0,30)(20,10)
\qbezier(0,10)(0,10)(7,17)
\qbezier(20,30)(20,30)(13,23)
\qbezier(20,10)(30,10)(30,20)
\qbezier(30,20)(30,20)(30,110)
\qbezier(30,110)(30,120)(10,120)
\qbezier(-20,10)(-30,10)(-30,20)
\qbezier(-30,20)(-30,20)(-30,110)
\qbezier(-30,110)(-30,120)(-10,120)
\put(15,90){\makebox(0,0)[cc]{\footnotesize $\alpha_{1}$}}
\put(15,70){\makebox(0,0)[cc]{\footnotesize $\alpha_{2}$}}
\put(15,50){\makebox(0,0)[cc]{\footnotesize $\alpha_{3}$}}
\put(-15,5){\makebox(0,0)[cc]{\footnotesize $\beta$}}
\put(15,5){\makebox(0,0)[cc]{\footnotesize $\gamma$}}
\end{picture}} 
\end{picture}
\caption{Oriented virtual diagram $D^{3}$.} \label{fig8}
\end{figure}
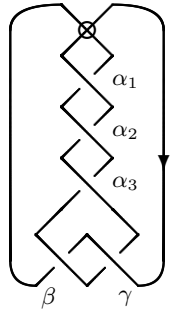
It is easy to check that all calculations for the diagram $D^{3}$ give the same results as for the diagram $D^{1}$ with the same behavior at crossings $\beta$ and crossings $\gamma$, respectively; and the same behavior at $\alpha_{1}, \alpha_{2}, \alpha_{3}$ of $D^{3}$ as at $\alpha_{1}$ of $D^{1}$. Similar arguments work for an arbitrary odd $k$. Since we get that for any odd $k \geq 1$ for the oriented virtual knot diagram $D^{k}$ we have $F^{1}_{D^{k}} (t, \ell) = P_{D^{k}} (t) = -t + 2 -t^{-1}$. 
 \end{proof}
 

\begin{table}[ht]
\caption{Table of diagrams, part 1} \label{table-1}
\begin{tabular}{|c|c|c|c|c|c| } \hline
 \includegraphics[width=2.cm]{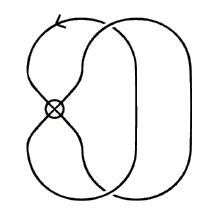} & 
 \includegraphics[width=2.cm]{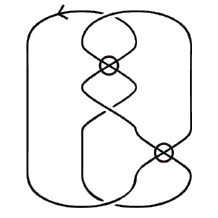} & 
\includegraphics[width=2.cm]{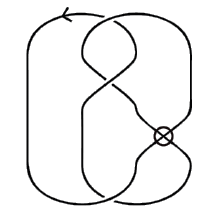} &
\includegraphics[width=2.cm]{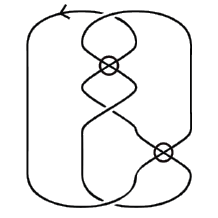} & 
\includegraphics[width=2.cm]{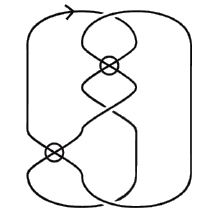} \\ 
2.1 & 3.1 & 3.2 & 3.3 & 3.4 \\ \hline
\includegraphics[width=2.cm]{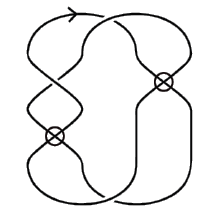} &
\includegraphics[width=2.cm]{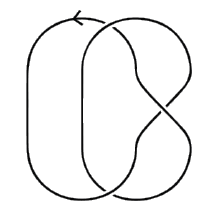} &
\includegraphics[width=2.cm]{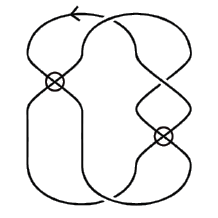} &
\includegraphics[width=2.cm]{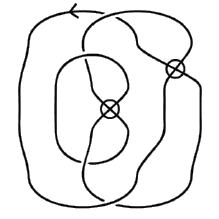} &
\includegraphics[width=2.cm]{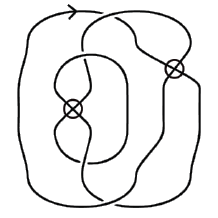} \\ 
3.5 & 3.6 & 3.7 & 4.1 & 4.2 \\ \hline
\includegraphics[width=2.cm]{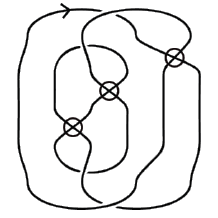} &
\includegraphics[width=2.cm]{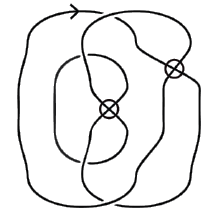} &
\includegraphics[width=2.cm]{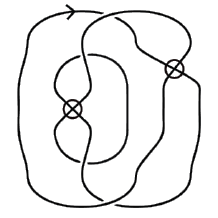} &
\includegraphics[width=2.cm]{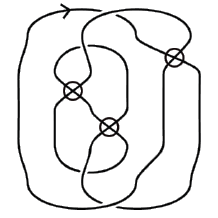} &
\includegraphics[width=2.cm]{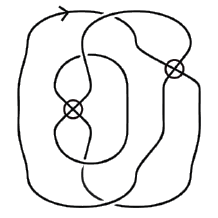} \\ 
4.3 & 4.4 & 4.5 & 4.6 & 4.7 \\ \hline
\includegraphics[width=2.cm]{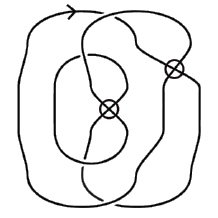} &
\includegraphics[width=2.cm]{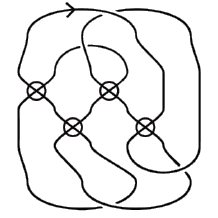} &
\includegraphics[width=2.cm]{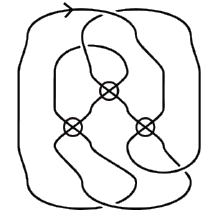} &
\includegraphics[width=2.cm]{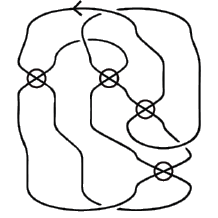} &
\includegraphics[width=2.cm]{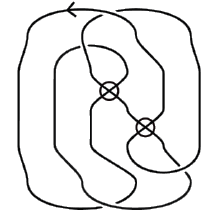} \\ 
4.8 & 4.9 & 4.10 & 4.11 & 4.12 \\ \hline
\includegraphics[width=2.cm]{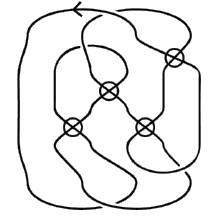} &
\includegraphics[width=2.cm]{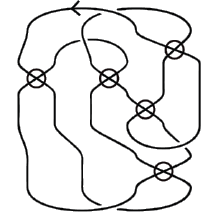} &
\includegraphics[width=2.cm]{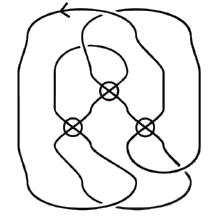} &
\includegraphics[width=2.cm]{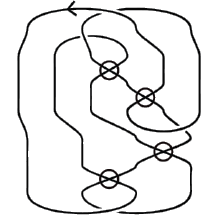} &
\includegraphics[width=2.cm]{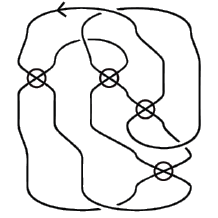} \\ 
4.13 & 4.14 & 4.15 & 4.16 & 4.17 \\ \hline
\includegraphics[width=2.cm]{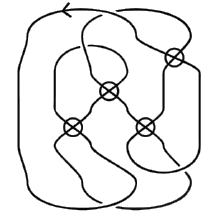} &
\includegraphics[width=2.cm]{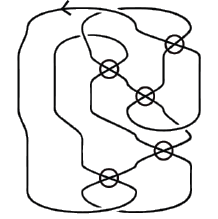} &
\includegraphics[width=2.cm]{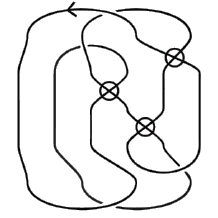} &
\includegraphics[width=2.cm]{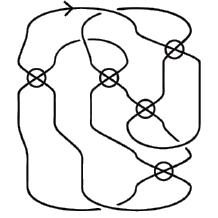} &
\includegraphics[width=2.cm]{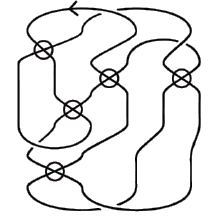} \\ 
4.18 & 4.19 & 4.20 & 4.21 & 4.22 \\ \hline
\includegraphics[width=2.cm]{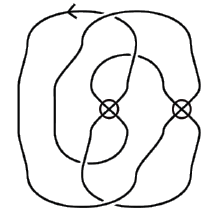} &
\includegraphics[width=2.cm]{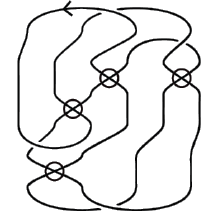} &
\includegraphics[width=2.cm]{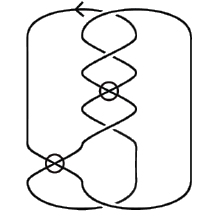} &
\includegraphics[width=2.cm]{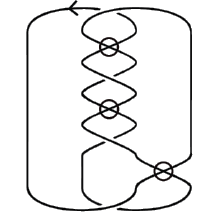} &
\includegraphics[width=2.cm]{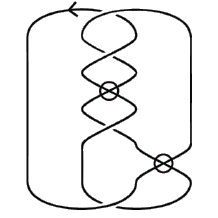} \\ 
4.23 & 4.24 & 4.25 & 4.26 & 4.27 \\ \hline 
\includegraphics[width=2.cm]{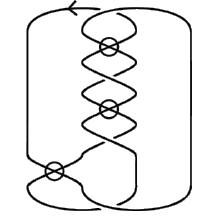} &
\includegraphics[width=2.cm]{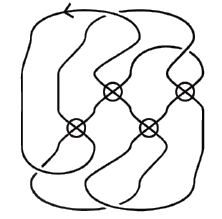} &
\includegraphics[width=2.cm]{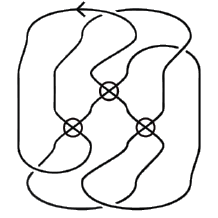} &
\includegraphics[width=2.cm]{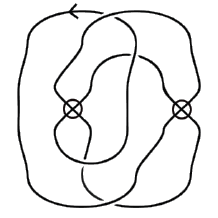} &
\includegraphics[width=2.cm]{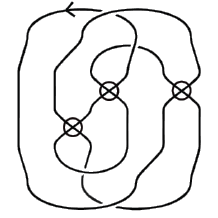} \\ 
4.28 & 4.29 & 4.30 & 4.31 & 4.32 \\ \hline 
\end{tabular}
\end{table}

\begin{table}[ht]
\caption{Table of diagrams, part 2} \label{table-2}
\begin{tabular}{|c|c|c|c|c| } \hline
\includegraphics[width=2.cm]{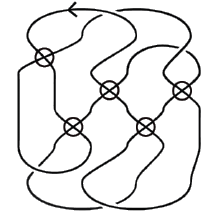} &
\includegraphics[width=2.cm]{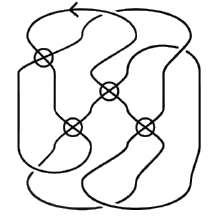} &
\includegraphics[width=2.cm]{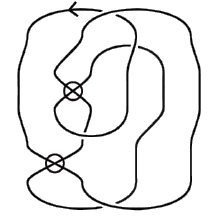} &
\includegraphics[width=2.cm]{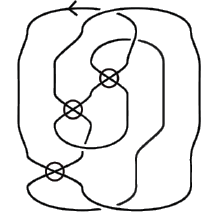} &
\includegraphics[width=2.cm]{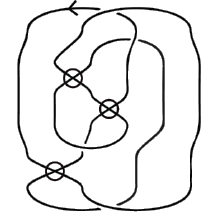} \\ 
  4.33 & 4.34 & 4.35 & 4.36 & 4.37 \\ \hline 
\includegraphics[width=2.cm]{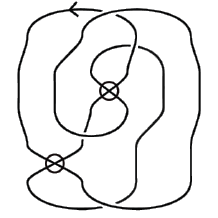} &
\includegraphics[width=2.cm]{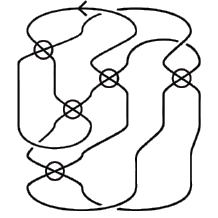} &
\includegraphics[width=2.cm]{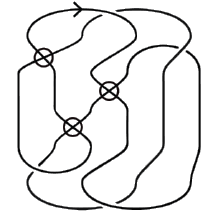} &
\includegraphics[width=2.cm]{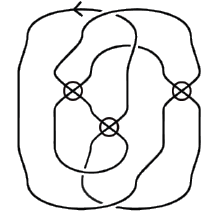} &
\includegraphics[width=2.cm]{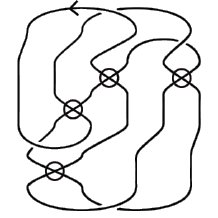} \\ 
  4.38 & 4.39 & 4.40 & 4.41 & 4.42 \\ \hline 
 \includegraphics[width=2.cm]{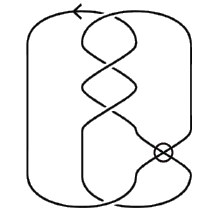} & 
 \includegraphics[width=2.cm]{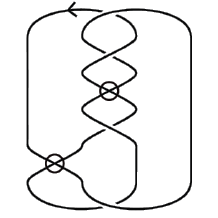} & 
\includegraphics[width=2.cm]{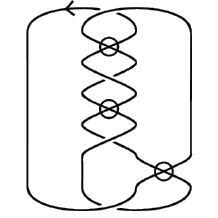} &
\includegraphics[width=2.cm]{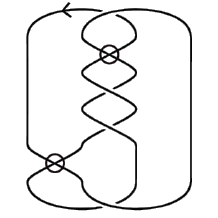} & 
\includegraphics[width=2.cm]{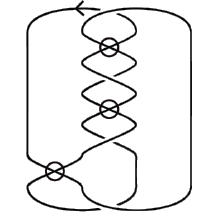} \\ 
  4.43 & 4.44 & 4.45 & 4.46 & 4.47 \\ \hline 
\includegraphics[width=2.cm]{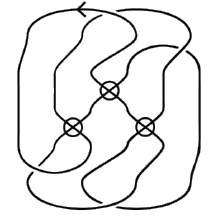} &
\includegraphics[width=2.cm]{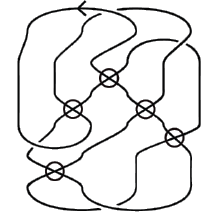} &
\includegraphics[width=2.cm]{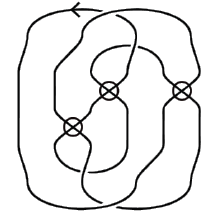} &
\includegraphics[width=2.cm]{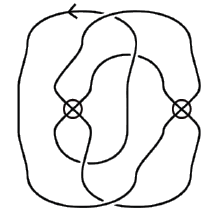} &
\includegraphics[width=2.cm]{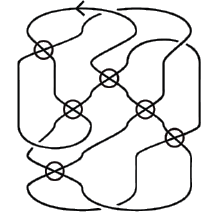} \\ 
  4.48 & 4.49 & 4.50 & 4.51 & 4.52 \\ \hline 
\includegraphics[width=2.cm]{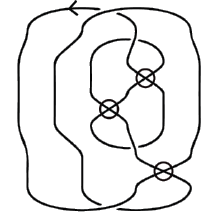} &
\includegraphics[width=2.cm]{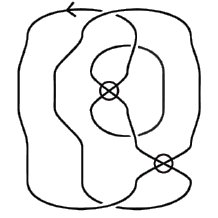} &
\includegraphics[width=2.cm]{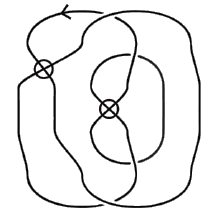} &
\includegraphics[width=2.cm]{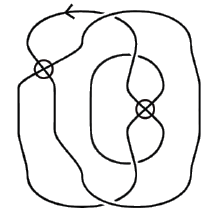} &
\includegraphics[width=2.cm]{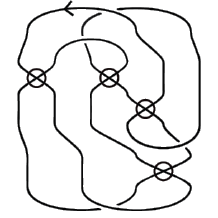} \\ 
  4.53 & 4.54 & 4.55 & 4.56 & 4.57 \\ \hline 
\includegraphics[width=2.cm]{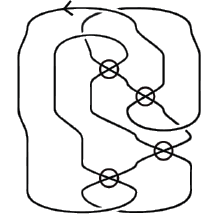} &
\includegraphics[width=2.cm]{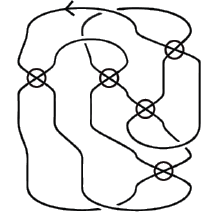} &
\includegraphics[width=2.cm]{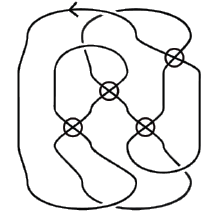} &
\includegraphics[width=2.cm]{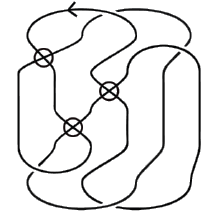} &
\includegraphics[width=2.cm]{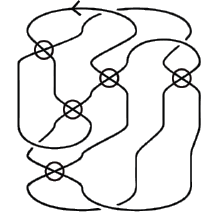} \\ 
  4.58 & 4.59 & 4.60 & 4.61 & 4.62 \\ \hline 
\includegraphics[width=2.cm]{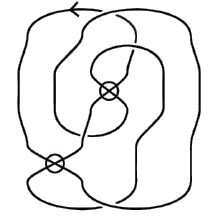} &
\includegraphics[width=2.cm]{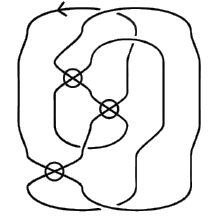} &
\includegraphics[width=2.cm]{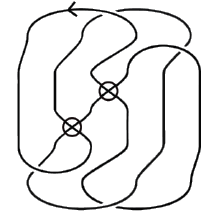} &
\includegraphics[width=2.cm]{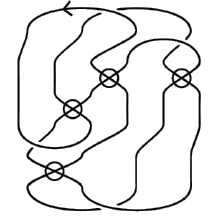} &
\includegraphics[width=2.cm]{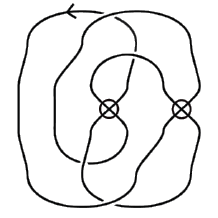} \\ 
  4.63 & 4.64 & 4.65 & 4.66 & 4.67 \\ \hline
\includegraphics[width=2.cm]{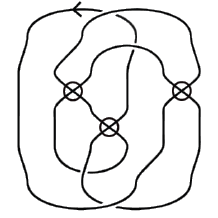} &
\includegraphics[width=2.cm]{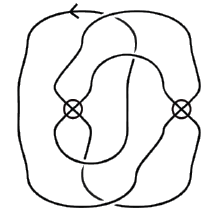} &
\includegraphics[width=2.cm]{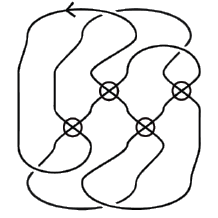} &
\includegraphics[width=2.cm]{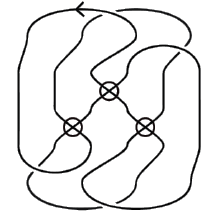} &
\includegraphics[width=2.cm]{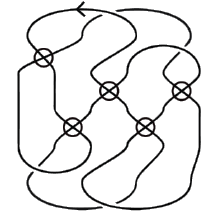} \\ 
 4.68 & 4.69 & 4.70 & 4.71 & 4.72 \\ \hline 
\end{tabular}
\end{table}

\begin{table}[ht]
\caption{Table of diagrams, part 3} \label{table-3}
\begin{tabular}{|c|c|c|c|c| } \hline
\includegraphics[width=2.cm]{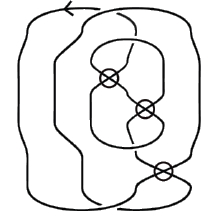} &
\includegraphics[width=2.cm]{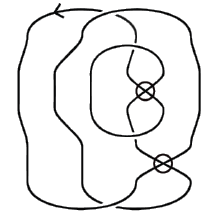} &
\includegraphics[width=2.cm]{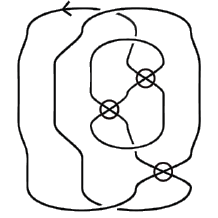} &
\includegraphics[width=2.cm]{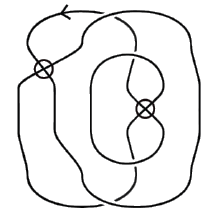} &
\includegraphics[width=2.cm]{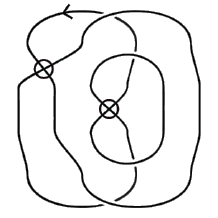} \\ 
  4.73 & 4.74 & 4.75 & 4.76 & 4.77 \\ \hline 
\includegraphics[width=2.cm]{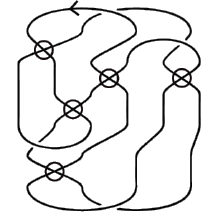} &
\includegraphics[width=2.cm]{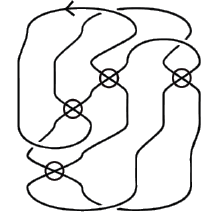} &
\includegraphics[width=2.cm]{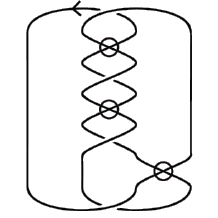} &
\includegraphics[width=2.cm]{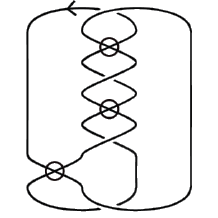} &
\includegraphics[width=2.cm]{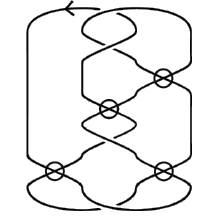} \\ 
  4.78 & 4.79 & 4.80 & 4.81 & 4.82 \\ \hline 
 \includegraphics[width=2.cm]{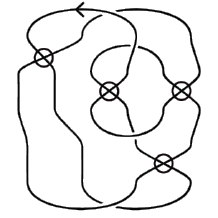} & 
 \includegraphics[width=2.cm]{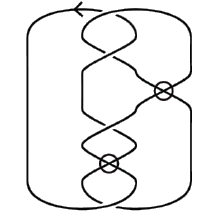} & 
\includegraphics[width=2.cm]{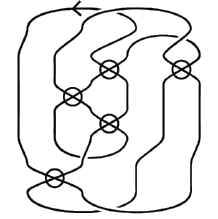} &
\includegraphics[width=2.cm]{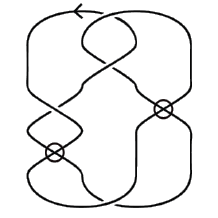} & 
\includegraphics[width=2.cm]{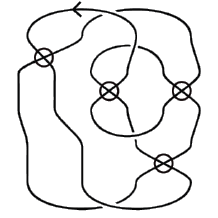} \\ 
  4.83 & 4.84 & 4.85 & 4.86 & 4.87 \\ \hline 
\includegraphics[width=2.cm]{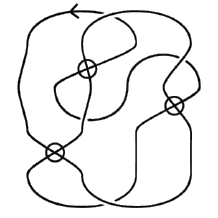} &
\includegraphics[width=2.cm]{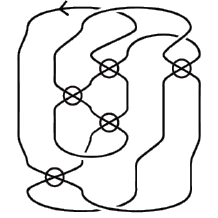} &
\includegraphics[width=2.cm]{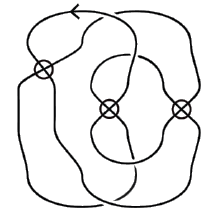} &
\includegraphics[width=2.cm]{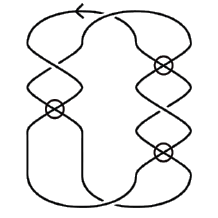} &
\includegraphics[width=2.cm]{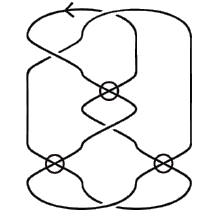} \\ 
  4.88 & 4.89 & 4.90 & 4.91 & 4.92 \\ \hline 
\includegraphics[width=2.cm]{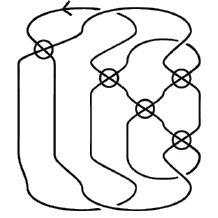} &
\includegraphics[width=2.cm]{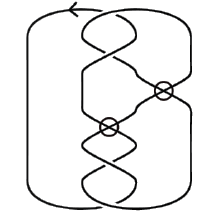} &
\includegraphics[width=2.cm]{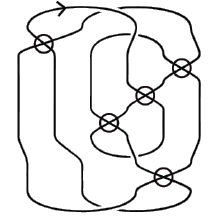} &
\includegraphics[width=2.cm]{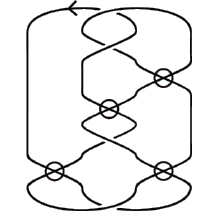} &
\includegraphics[width=2.cm]{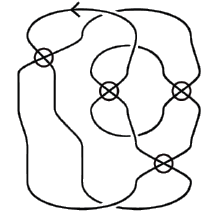} \\ 
  4.93 & 4.94 & 4.95 & 4.96 & 4.97 \\ \hline 
\includegraphics[width=2.cm]{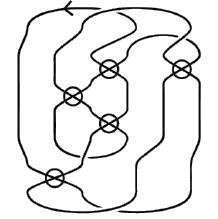} &
\includegraphics[width=2.cm]{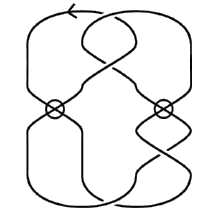} &
\includegraphics[width=2.cm]{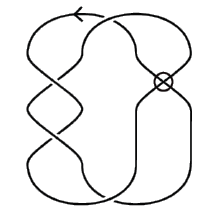} &
\includegraphics[width=2.cm]{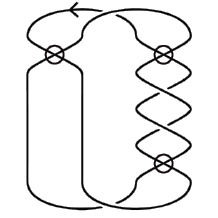} &
\includegraphics[width=2.cm]{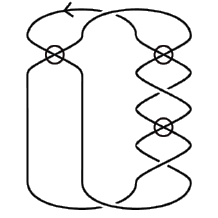} \\ 
  4.98 & 4.99 & 4.100 & 4.101 & 4.102 \\ \hline 
\includegraphics[width=2.cm]{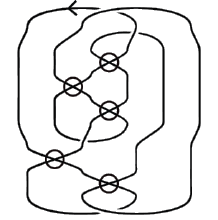} &
\includegraphics[width=2.cm]{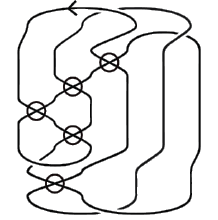} &
\includegraphics[width=2.cm]{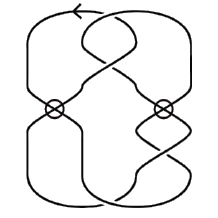} &
\includegraphics[width=2.cm]{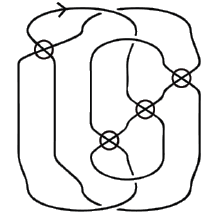} &
\includegraphics[width=2.cm]{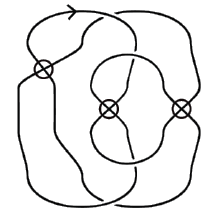} \\ 
  4.103 & 4.104 & 4.105 & 4.106 & 4.107 \\ \hline
\includegraphics[width=2.cm]{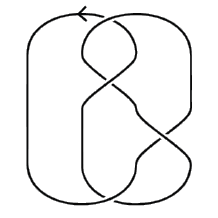} & & & & \\ 
 4.108 &  &  &  &  \\ \hline 
\end{tabular}
\end{table}

\begin{table}[ht]
\caption{Table of $F$-polynomials, part 1} \label{table-4} 
\begin{tabular}{|l|l|l|} \hline 
name of the virtual knot  & $n$ & $F^{n}(t, \ell)$-polynomials \\ \hline \hline 
\textbf{2.1}, \textbf{3.2}, \textbf{4.4}, \textbf{4.5}, \textbf{4.30}, \textbf{4.40},   &   1 & $ -t^{-1}+2-t$ \\ 
\textbf{4.54}, \textbf{4.61}, \textbf{4.69}, \textbf{4.74},  \textbf{4.94} & & \\  \hline 
\textbf{3.1} & 1 & $ -t^{-2}+t^{-1}+\ell^{-2}-t $ \\
& 2 & $  -t^{-2}+t^{-1}+\ell-t $ \\
& 3 & $  -t^{-2}+t^{-1}+1-t$  \\ \hline 
\textbf{3.3}  & 1 & $ -t^{-2}+3\ell^{-2}-2t$ \\
& 2 & $ -t^{-2}+3\ell-2t$ \\
& 3 & $ -t^{-2}+3-2t$ \\\hline 
\textbf{3.4}  & 1 & $ \ell^{-2}-2t+t^{2} $ \\
& 2 & $ \ell-2t+t^{2}$\\
& 3 & $1-2t+t^{2}$\\ \hline 
\textbf{3.5}, \textbf{3.7}, \textbf{4.65}, \textbf{4.85}, \textbf{4.86}. \textbf{4.106} &  1 &  $ -t^{-2}+2-t^{2} $ \\ \hline
\textbf{3.6}, \textbf{4.2}, \textbf{4.6}, \textbf{4.8}, \textbf{4.12}, \textbf{4.41}, \textbf{4.55}, \textbf{4.56}, \textbf{4.58},
& 1 &  $0$ \\ 
 \textbf{4.59},  \textbf{4.68}, \textbf{4.71}, \textbf{4.72}, \textbf{4.75}, \textbf{4.76}, \textbf{4.77}, & & \\   
\textbf{4.90}, \textbf{4.98}, \textbf{4.99}, \textbf{4.105}, \textbf{4.107}. \textbf{4.108} & & \\ \hline 
\textbf{4.1}, \textbf{4.3}, \textbf{4.7}, \textbf{4.43}, \textbf{4.53}, \textbf{4.73}, \textbf{4.100}  & 1 &  $  -2t^{-1}+4-2t $ \\ \hline 
\textbf{4.9}, \textbf{4.27}, \textbf{4.33}, \textbf{4.44}  & 1& $  -t^{-1}\ell^2+2-t\ell^{-2} $\\
& 2& $ -t^{-1}\ell^{-1}+2-t\ell $ \\
& 3& $ -t^{-1}+2-t$ \\ \hline 
\textbf{4.10} & 1 & $ -t^{-1}-\ell^{-2}-1+3\ell^2+t\ell^{-2}-t^{2}$ \\
& 2& $ -t^{-1}+3\ell^{-1}-1-\ell+tl-t^{2}$ \\
& 3& $ -t^{-1}+1+t-t^{2}$ \\ \hline 
\textbf{4.11} & 1 & $ -t^{-2}\ell^2+2\ell^{-2}+\ell^2-t\ell^{-2}-t $ \\
& 2& $  -t^{-2}\ell^{-1}+\ell^{-1}+2\ell-t-t\ell$ \\
& 3& $ -t^{-2}+3-2t $\\ \hline 
\textbf{4.13} & 1 & $ -t^{-1}\ell^{-2}+t^{-1}\ell^2 $ \\
& 2 & $ t^{-1}\ell^{-1}-t^{-1}\ell $\\
& 3& $ 0$ \\ \hline 
\textbf{4.14} & 1 & $ t^{-2}\ell^2-t^{-1}-t+t^{2} \ell^{-2} $ \\
& 2 & $ t^{-2} \ell^{-1}-t^{-1}-t+t^{2} \ell $\\
& 3 & $ t^{-2}-t^{-1}-t+t^{2} $ \\ \hline 
\textbf{4.15} & 1 & $ -t^{-2}+3 \ell^{-2}-1+\ell^2-t-t\ell^2 $\\
&2& $  -t^{-2}+\ell^{-1}-1+3\ell-t\ell^{-1}-t $ \\
&3 & $ -t^{-2}+3-2t $ \\ \hline 
\textbf{4.16} & 1 & $  -\ell^{-2} + 2 - \ell^{2} $ \\
& 2& $ -\ell^{-1} + 2 - \ell $ \\
& 3 & $  0 $ \\ \hline 
\textbf{4.17}, \textbf{4.57} & 1 & $ -t^{-2}\ell^2+t^{-1}+\ell^2-t\ell^{-2} $ \\
& 2 & $ -t^{-2}\ell^{-1}+t^{-1}+\ell^{-1}-t\ell $\\
& 3 & $ -t^{-2}+t^{-1}+1-t $ \\ \hline 
\textbf{4.18}, \textbf{4.60}  &  1 & $ -t^{-1}\ell^{-2}+2-t\ell^{-2}  $ \\
& 2 & $  -t^{-1}\ell+2-t\ell $ \\
& 3 & $  -t^{-1}+2-t $ \\ \hline 
\textbf{4.19} & 1 & $  -t^{-1}+\ell^2+t\ell^{-2}-t^{2}\ell^{-2} $\\
& 2 & $ -t^{-1}+\ell^{-1}+t\ell-t^{2}\ell $ \\
& 3 & $ -t^{-1}+1+t-t^{2} $ \\ \hline 
\end{tabular}
\end{table}

\begin{table}[ht]
\caption{Table of $F$-polynomials, part 2} \label{table-5}
\begin{tabular}{|l|l|l|} \hline
name of the virtual knot & $n$ & $F^{n}(t, \ell)$-polynomials \\ \hline \hline 
\textbf{4.20} &  1 & $  t^{-2}-t^{-1}-t^{-1}\ell^2+1 $ \\
& 2 & $ t^{-2}-t^{-1}\ell^{-1}-t^{-1}+1 $ \\
& 3 & $  t^{-2}-2t^{-1}+1 $ \\ \hline 
\textbf{4.21} & 1 & $ t^{-2}\ell^{-2}-2+t^{2}\ell^2 $ \\
& 2 & $ t^{-2}\ell-2+t^{2}\ell^{-1}$ \\
& 3 & $ t^{-2}-2+t^{2}$ \\ \hline 
\textbf{4.22} & 1 & $ -t^{-1}+1+t^{2}-t^{3}\ell^{-2} $ \\
& 2 & $ -t^{-1}+1+t^{2}-t^{3}\ell $ \\
& 3 & $  -t^{-1}+1+t^{2}-t^{3} $ \\ \hline 
\textbf{4.23} &  1 & $ -t^{-1}-1+2\ell^2+t\ell^2-t^{2}$ \\
& 2 & $ -t^{-1}+2\ell^{-1}-1+t\ell^{-1}-t^{2}$ \\
& 3 & $ -t^{-1}+1+t-t^{2} $ \\  \hline 
\textbf{4.24} & 1 & $ t^{-2}-2\ell^{-1}+1-t+t^{3}\ell^{-2} $ \\
& 2 & $ t^{-2}-\ell^{-1}+1-\ell-t+t^{3} \ell $ \\
& 3 & $  t^{-2}+1-2\ell-t+t^{3} $ \\
& 4 & $ t^{-2}-1-t+t^{3} $ \\  \hline 
\textbf{4.25} &  1 & $  -t^{-1}-t^{-1}\ell^2+4-t\ell^{-2}-t $ \\
& 2 & $  -t^{-1}\ell^{-1}-t^{-1}+4-t-t\ell $ \\
& 3 & $  -2t^{-1}+4-2t $ \\ \hline 
\textbf{4.26} & 1 &  $  -t^{-3}+t^{-1}+t^{-1}\ell^2-t\ell^2 $ \\
& 2 & $  -t^{-3}+t^{-1}\ell^{-1}+t^{-1}-t\ell^{-1} $ \\
& 3 & $  -t^{-3}+2t^{-1}-t $ \\ \hline 
\textbf{4.28} & 1 & $  t^{-3}-t^{-1}\ell^{-2}-2\ell^3+t\ell^{-2}+t $ \\
& 2 & $  t^{-3}-t^{-1}\ell-2+t+t\ell $ \\
& 3 & $  t^{-3}-t^{-1}-2l^{-1}+2t $ \\
& 4 & $  t^{-3}-t^{-1}-2+2t $ \\ \hline 
\textbf{4.29} & 1 & $ -t^{-1}\ell^{-2}-t^{-1}+\ell^{-2}+2\ell^2-t^{2}\ell^2 $ \\
& 2 & $ -t^{-1}-t^{-1}\ell+2l^{-1}+l-t^{2}\ell^{-1} $ \\
& 3 & $  -2t^{-1}+3-t^{2} $ \\ \hline 
\textbf{4.31}, \textbf{4.51} & 1 & $  t\ell^{-2}-t\ell^2 $ \\
& 2 & $  -t\ell^{-1}+t\ell $ \\
& 3 & $ 0 $ \\ \hline 
\textbf{4.32} &  1 & $  -t^{-2}+t^{-1}+1-t\ell^{-2} $ \\
& 2 & $ -t^{-2}+t^{-1}+1-t\ell $ \\
& 3 & $  -t^{-2}+t^{-1}+1-t $ \\
\hline 
\textbf{4.34} & 1 & $  \ell^{-2}-t-t\ell^2+t^{2}\ell^2 $ \\
& 2 & $  \ell-t\ell^{-1}-t+t^{2}\ell^{-1} $ \\
& 3 & $  1-2t+t^{2} $ \\ \hline 
\textbf{4.35} &  1 & $  -t^{-1}+1+t \ell^2-t^{2} $ \\
& 2 & $ -t^{-1}+1+t \ell^{-1}-t^{2} $ \\
& 3 & $  -t^{-1}+1+t-t^{2} $ \\ \hline 
\textbf{4.36} & 1 & $  t^{-2}-2+t^{2} $ \\ \hline 
\textbf{4.37} & 1 & $   -t^{-2}-t^{-1}+4-t-t^{2} $ \\ \hline 
\textbf{4.38} & 1 & $ \ell^{-2}-1+ \ell^2-t-t\ell^2+t^{2} $ \\
& 2 & $  \ell^{-1}-1+ \ell-t \ell^{-1}-t+t^{2} $ \\
& 3 & $ 1-2t+t^{2} $ \\ \hline 
\end{tabular}
\end{table} 

\begin{table}[ht]
\caption{Table of $F$-polynomials, part 3} \label{table-6}
\begin{tabular}{|l|l|l|} \hline
name of the virtual knot & $n$ & $F^{n}(t, \ell)$-polynomials \\ \hline \hline 
 \textbf{4.39} & 1 & $ -t^{3} \ell^{-2} -t^{-1} -1 +t^{2} + 2 \ell $ \\
& 2 & $ -t^{-1} -1 +t^{2} + 2 \ell - t^{3} \ell $ \\
& 3 & $ 2 \ell^{-1} - t^{-1} - 1 + t^{2} - t^{3}  $ \\ 
& 4 & $ -t^{-1} + 1 + t^{2} -  t^{3}  $ \\ \hline
\textbf{4.42} &  1 & $  1-t-t^{2}+t^{3} \ell^{-2} $ \\
& 2 & $  \ell^{-1}+1-\ell-t-t^{2}+t^{3} \ell $ \\
&3 & $  1-t-t^{2}+t^{3} $ \\ \hline 
\textbf{4.45} &  1 & $  -t^{-3}+t^{-1}+2\ell^{-3}-t\ell^{-2}-t\ell^2 $ \\
& 2 & $  -t^{-3}+t^{-1}+2-t\ell^{-1}-t\ell $ \\
& 3 & $  -t^{-3}+t^{-1}+2\ell-2t $ \\
& 4 & $  -t^{-3}+t^{-1}+2-2t $ \\ \hline 
\textbf{4.46} & 1 & $  -t^{-1}+t^{-1} \ell^2+t \ell^{-2}-t $ \\
& 2 & $  t^{-1} \ell^{-1}-t^{-1}-t+t\ell $ \\
& 3 & $ 0 $ \\ \hline 
\textbf{4.47} & 1 & $  t^{-3}-t^{-1} \ell^{-2}-t^{-1} \ell^2+t $ \\
& 2 & $  t^{-3}-t^{-1} \ell^{-1}-t^{-1} \ell+t $ \\
& 3 & $  t^{-3}-2t^{-1}+t $ \\ \hline 
\textbf{4.48}, \textbf{4.82}  &  1 & $  -t^{-2} \ell^{-2}-t^{-1}+4-t-t^{2} \ell^2 $ \\
& 2 & $  -t^{-2}\ell-t^{-1}+4-t-t^{2} \ell^{-1} $ \\
& 3 & $  -t^{-2}-t^{-1}+4-t-t^{2} $ \\\hline 
 \textbf{4.49} & 1 & $  t^{-2} \ell^{-2}-t^{-1} \ell^{-2}-t^{-1}+ \ell^2 $ \\
& 2 & $  t^{-2} \ell-t^{-1}-t^{-1} \ell+ \ell^{-1} $ \\
& 3 & $  t^{-2}-2t^{-1}+1 $ \\ \hline 
\textbf{4.50} & 1 & $  -t^{-2}+t^{-1}+2\ell^{-2}-1-t\ell^{-2} $ \\
& 2 & $  -t^{-2}+t^{-1}-1+2\ell-t\ell $ \\
& 3 & $ -t^{-2}+t^{-1}+1-t $ \\ \hline 
\textbf{4.52} &  1 & $ -t^{-1}\ell^2+\ell^{-2}+\ell^2-t\ell^{-2} $ \\
& 2 & $  -t^{-1}\ell^{-1}+\ell^{-1}+\ell-t\ell $ \\
& 3 & $  -t^{-1}+2-t $ \\ \hline 
\textbf{4.62} &  1 & $  -t^{-2}-t^{-1}+1+2\ell-t^{3}\ell^{-2} $ \\
& 2 & $  -t^{-2}-t^{-1}+1+2\ell-t^{3}\ell $ \\
& 3 & $  -t^{-2}-t^{-1}+2\ell^{-1}+1-t^{3} $ \\
& 4 & $  -t^{-2}-t^{-1}+3-t^{3} $ \\ \hline 
\textbf{4.63} &  1 & $  -t^{-2}+\ell^{-2}+1+\ell^2-t-t\ell^2 $ \\
& 2 & $  -t^{-2}+\ell^{-1}+1+\ell-t\ell^{-1}-t $ \\
& 3 & $  -t^{-2}+3-2t $ \\ \hline 
\textbf{4.64} & 1 & $  t^{-2}-t^{-1}-t+t^{2} $ \\ \hline 
\textbf{4.66} & 1 & $  t^{-2}-1-t+t^{3}\ell^{-2} $ \\
& 2 & $  t^{-2}+\ell^{-1}-1-\ell-t+t^{3}\ell $ \\
& 3 & $  t^{-2}-1-t+t^{3} $ \\ \hline 
\textbf{4.67} &  1 & $  t^{-2}-t^{-1}-1+t\ell^2 $ \\
& 2 & $  t^{-2}-t^{-1}-1+t\ell^{-1} $ \\
& 3 & $  t^{-2}-t^{-1}-1+t $ \\ \hline 
\textbf{4.70} & 1 & $ -t^{-1}+\ell^2+t\ell^2-t^{2}\ell^2 $ \\
& 2 & $  -t^{-1}+\ell^{-1}+t\ell^{-1}-t^{2}\ell^{-1} $ \\
& 3 & $  -t^{-1}+1+t-t^{2} $ \\ \hline 
\end{tabular}
\end{table}

\begin{table}[ht]
\caption{Table of $F$-polynomials, part 4} \label{table-7}
\begin{tabular}{|l|l|l|} \hline
name of the virtual knot & $n$ & $F^{n}(t, \ell)$-polynomials \\ \hline \hline 
\textbf{4.78} & 1 & $  -t^{-2}-t^{-1}-1+4\ell-t^{3}\ell^{-2} $ \\
&  2 & $  -t^{-2}-t^{-1}-1+4\ell-t^{3}\ell $ \\
& 3 &  $ -t^{-2}-t^{-1}+4\ell^{-1}-1-t^{3} $ \\
& 4 & $  -t^{-2}-t^{-1}+3-t^{3} $ \\ \hline 
\textbf{4.79} & 1 & $  2\ell^{-1}-1-t-t^{2}+t^{3}\ell^{-2} $ \\
& 2 & $  3\ell^{-1}-1-\ell-t-t^{2}+t^{3}\ell $ \\
& 3 & $ -1+2\ell-t-t^{2}+t^{3} $ \\ 
& 4 & $ 1-t-t^{2}+t^{3} $ \\ \hline 
\textbf{4.80} & 1 & $  -t^{-3}+4\ell^{-3}-t\ell^{-2}-t-t\ell^2 $ \\
& 2 & $  -t^{-3}+4-t\ell^{-1}-t-t\ell $ \\
& 3 & $  -t^{-3}+4\ell-3t $ \\
& 4 & $  -t^{-3}+4-3t $ \\ \hline 
\textbf{4.81} & 1& $  t^{-3}-t^{-1}\ell^{-2}-t^{-1}-t^{-1}\ell^2+2\ell^3 $ \\
& 2 & $ t^{-3}-t^{-1}\ell^{-1}-t^{-1}-t^{-1}\ell+2 $ \\
& 3 & $  t^{-3}-3t^{-1}+2\ell^{-1} $ \\
& 4 & $  t^{-3}-3t^{-1}+2 $ \\ \hline 
\textbf{4.83} & 1 & $  -t^{-3}+t^{-2}\ell^{-2}-t^{-1}+2\ell-t^{2}\ell^{-2} $ \\
& 2 & $  -t^{-3}+t^{-2}\ell-t^{-1}+2\ell^{-2}-t^{2}\ell $ \\
& 3 & $  -t^{-3}+t^{-2}-t^{-1}+2\ell-t^{2} $ \\
& 4 & $  -t^{-3}+t^{-2}-t^{-1}+2-t^{2} $ \\ \hline 
\textbf{4.84} & 1 & $  -t^{-2}\ell^{-2}+t^{-1}+t-t^{2}\ell^2 $ \\
& 2 &  $  -t^{-2}l+t^{-1}+t-t^{2}\ell^{-1} $ \\
& 3 & $  -t^{-2}+t^{-1}+t-t^{2}$ \\  \hline 
\textbf{4.87} & 1 & $  -t^{-3}-t^{-1}+4\ell-t^{2}\ell^{-2}-t^{2}\ell^2 $ \\
& 2 & $  -t^{-3}-t^{-1}+4\ell^{-2}-t^{2}\ell^{-1}-t^{2}\ell $ \\
& 3 & $  -t^{-3}-t^{-1}+4\ell-2t^{2} $ \\
& 4 & $  -t^{-3}-t^{-1}+4-2t^{2} $ \\ \hline 
\textbf{4.88} & 1 & $ t^{-3}-t^{-2}\ell^{-2}-t^{-2}\ell^2+t^{-1} $ \\
& 2 & $  t^{-3}-t^{-2}\ell^{-1}-t^{-2}\ell+t^{-1} $ \\
& 3 & $ t^{-3}-2t^{-2}+t^{-1} $ \\ \hline 
\textbf{4.89} & 1 & $ -2t^{-2}+4-2t^{2} $ \\ \hline 
\textbf{4.91} & 1 & $ -t^{-3}-t^{-1}+4-t-t^{3} $ \\ \hline 
\textbf{4.92}, \textbf{4.95}, \textbf{4.101} & 1 & $ -t^{-3}+2-t^{3} $ \\ \hline 
\textbf{4.93} & 1 & $ t^{-3}-t^{-2}\ell^{-2}-t^{-2}\ell^2+2\ell^{-1}-t $ \\
& 2 & $  t^{-3}-t^{-2}\ell^{-1}-t^{-2}\ell+2\ell^2-t $ \\
& 3 & $  t^{-3}-2t^{-2}+2\ell^{-1}-t $ \\
& 4 & $  t^{-3}-2t^{-2}+2-t $ \\ \hline 
\textbf{4.96} & 1 & $  -t^{-2}\ell^{-2}+2-t^{2}\ell^2 $ \\
& 2 & $ -t^{-2}\ell+2-t^{2}\ell^{-1} $ \\
& 3 & $  -t^{-2}+2-t^{2} $ \\  \hline 
\textbf{4.97} &  1 & $  t^{-2}\ell^{-2}-t^{-1}-t^{2}\ell^{-2}+t^{3} $ \\
& 2 & $  t^{-2}\ell-t^{-1}-t^{2}\ell+t^{3} $ \\
& 3 & $ t^{-2}-t^{-1}-t^{2}+t^{3} $ \\ \hline 
\textbf{4.102} & 1 & $ t^{-3}-t^{-1}-t+t^{3} $ \\  \hline 
\end{tabular}
\end{table}

\begin{table}[ht]
\caption{Table of $F$-polynomials, part 5} \label{table-8}
\begin{tabular}{|l|l|l|} \hline
name of the virtual knot & $n$ & $F^{n}(t, \ell)$-polynomials \\ \hline \hline 
\textbf{4.103} &1 & $  -t^{-2}\ell^{-2}-t^{-2}\ell^2+t^{-1}+2\ell^{-1}-t^{3} $\\
& 2 & $  -t^{-2}\ell^{-1}-t^{-2}\ell+t^{-1}+2\ell^2-t^{3} $ \\
& 3 & $  -2t^{-2}+t^{-1}+2\ell^{-1}-t^{3} $ \\
& 4 & $  -2t^{-2}+t^{-1}+2-t^{3} $ \\ \hline 
\textbf{4.104} & 1 & $ \ t^{-3}-2+t^{3} $ \\  \hline 
\end{tabular}
\end{table}



\begin{thebibliography}{10}

\bibitem{cheng2013polynomial}
Z.~Cheng, H.~Gao, \emph{A polynomial invariant of virtual links}, Journal of Knot Theory and Its Ramifications  \textbf{22(12)} (2013), 1341002.

\bibitem{conway1970enumeration}
J.~Conway, \emph{An enumeration of knots and links, and some of their algebraic properties}, in: Computational Problems in Abstract Algebra (Proc. Conf., Oxford, 1967), 1970, 329--358. 

\bibitem{dye}
H.~Dye, \emph{An invitation to knot theory: virtual and classical.} 2016.

\bibitem{dye2009virtual}
H.~Dye, L.~Kauffman, \emph{Virtual crossing number and the arrow polynomial}, Journal of Knot Theory and Its Ramifications \textbf{18(10)} (2009), 1335--1357.

\bibitem{folwaczny2013linking}
L.~Folwaczny, L.~Kauffman, \emph{A linking number definition of the affine index polynomial and applications}, Journal of Knot Theory and Its Ramifications \textbf{22(12)} (2013), 1341004.

\bibitem{im2010index}
Y.-H.~Im, K.~Lee, S.-Y.~Lee, \emph{Index polynomial invariant of virtual links}, Journal of Knot Theory and Its Ramifications \textbf{19(5)} (2010), 709--725.

\bibitem{jeong2016zero}
M.-J.~Jeong, \emph{A zero polynomial of virtual knots}, Journal of Knot Theory and Its Ramifications, \textbf{25(1)} (2016), 1550078. 

\bibitem{Green}
J.~Green, \emph{A table of virtual knots}, available at https://www.math.toronto.edu /drorbn/Students/GreenJ/ last updated August 10, 2004.  

\bibitem{kauffman1999virtual}
L.~Kauffman, \emph{Virtual knot theory}, European Journal of Combinatorics \textbf{20(7)} (1999), 663--691. 

\bibitem{kauffman2009extended}
L.~Kauffman, \emph{An extended bracket polynomial for virtual knots and links}, Journal of Knot Theory and Its Ramifications \textbf{18(10)} (2009),1369--1422.

\bibitem{kauffman2013affine}
L.~Kauffman, \emph{An affine index polynomial invariant of virtual knots}, Journal of Knot Theory and Its Ramifications \textbf{22(04)} (2013), 1340007. 

\bibitem{kauffman2018cobordism}
L.~Kauffman, \emph{Virtual knot cobordism and the affine index polynomial}, Journal of Knot Theory and Its Ramifications, \textbf{27(11)} (2018), 1843017. 

\bibitem{KPV} 
K.~Kaur, M.~Prabhakar, A.~Vesnin, \emph{Two-variable polynomial invariants of virtual knots arising from flat virtual knot invariants}, Journal of Knot Theory and Its Ra\-mi\-fications, \textbf{27(13)} (2018), 1842015. 

\bibitem{satoh2014writhes}
S.~Satoh, K.~Taniguchi, \emph{The writhes of a virtual knot}, Fundamenta Mathematicae \textbf{225} (2014), 327--341.

\end{thebibliography}
\end{document}